\documentclass{amsproc}
 \usepackage{amssymb,amsmath,amsthm,enumerate,amscd,graphicx,color}
\usepackage{pdfsync}
 \usepackage[colorlinks=true, pdfstartview=FitV, linkcolor=blue, citecolor=blue, urlcolor=blue]{hyperref}

\numberwithin{equation}{section}
\newtheorem{thm}{Theorem}[subsection]
\newtheorem{lem}[thm]{Lemma}

\newtheorem{cor}[thm]{Corollary}
\newtheorem{prop}[thm]{Proposition}
\newtheorem{rem}[thm]{Remark}

\newcommand{\Lnum}[3]{\left[ \begin{matrix} #1 \ ; \  #2 \\ \ \ #3\   \end{matrix} \right]}
\newcommand{\g}{\widehat{\mathfrak g}}
\newcommand{\D}{\Delta}
\newcommand{\Hom}{{\text {Hom}}}

\newcommand{\thmref}[1]{Theorem~\ref{#1}}
\newcommand{\propref}[1]{Proposition~\ref{#1}}
\newcommand{\corref}[1]{Corollary~\ref{#1}}
\newcommand{\secref}[1]{\S\ref{#1}}
\newcommand{\lemref}[1]{Lemma~\ref{#1}}
\newcommand{\eqnref}[1]{~(\ref{#1})}

\def\um{{\underline m}}
\def\ul{{\underline l}}
\def\up{{\underline p}}

\begin{document}

%\centerline{\sc Imaginary Verma modules and Kashiwara algebras for $U_q(\bighat{\mathfrak{g}})$.}

%\date{}                                           % Activate to display a given date or no date

\title{}
\title{\sc An imaginary PBW  basis for quantum affine algebras of type 1.}
\author{ Ben Cox}
\author{Vyacheslav Futorny}
\author{Kailash C. Misra}
\address{Department of Mathematics \\
The Graduate School at the College of Charleston \\
66 George Street  \\
Charleston SC 29424, USA}\email{coxbl@cofc.edu}

\address{Department of Mathematics\\
 University of S\~ao Paulo\\
 S\~ao Paulo, Brazil}
 \email{futorny@ime.usp.br}
 \address{Department of Mathematics\\
 North Carolina State University\\
 Raleigh, NC 27695-8205, USA}\email{misra@ncsu.edu}
 
 \begin{abstract}  Let $\widehat{\mathfrak g}$ be an affine Lie algebra of type 1.   We give a PBW basis for the quantum affine algebra $U_q(\widehat{\mathfrak g})$ with respect to the triangular decomposition of  $\widehat{\mathfrak g}$ associated with the imaginary positive root system.\end{abstract}
\date{}
%\thanks{
%The authors are grateful to the organizers for the invitation to the conference at Banff where this project was initiated. The first author would like to thank North Carolina State University for the support and hospitality during his numerous visits to Raleigh.  The second author was partially supported by Fapesp (processo 2005/60337-2) and CNPq (processo 301743/2007-0). He is grateful to the North Carolina State University for the support and hospitality during his visit to Raleigh. The third author was partially supported by the NSA grant H98230-08-1-0080.}

\thanks{The first author would like to thank North Carolina State University for the support and hospitality during his numerous visits to Raleigh.  The second author was partially supported by Fapesp (processo 2005/60337-2) and CNPq (processo 301743/2007-0). KCM was partially supported by the NSA grant H98230-12-1-0248.}

\maketitle

\section{Introduction}

Although affine Lie algebras are infinite dimensional analogs of finite dimensional semisimple Lie algebras, they have special features that do not have analogs in the finite dimensional theory. One such feature is the existence of  closed partitions of the root system into sets of positive and negative roots which are not equivalent under the action of the Weyl group to the standard partitions of the root system. Such partitions are called nonstandard partitions. The classification of closed subsets of the root system for affine Kac-Moody algebras was obtained by Jakobsen and Kac \cite{JK,MR89m:17032}, and independently by Futorny \cite{MR1078876,MR1175820}. In particular, it is shown that for affine Lie algebras there are only a finite number of Weyl-equivalency classes of these nonstandard partitions. Corresponding to each non-standard partition we have non-standard Borel subalgebras from which one may induce other non-standard Verma-type modules and these typically contain both finite and infinite dimensional weight spaces. For example, for the affine Lie  algebra $ \widehat{\mathfrak{sl}(2)}$, the only non-standard modules of Verma-type are the {\it imaginary Verma modules} \cite{MR95a:17030}. In this paper we focus on the imaginary Verma modules for affine Lie algebras of type 1.

Let $\widehat{\mathfrak g}$ be an affine Lie algebra of type 1 and $U_q(\widehat{\mathfrak g})$ denote the associated quantum affine algebra  introduced independently by Drinfeld \cite{MR802128} and Jimbo \cite{MR797001}. One of the problems in dealing with nonstandard partitions of root systems is that the associated triangular decomposition of $\widehat{\mathfrak g}$ can not be lifted to a triangular decomposition of $U_q(\widehat{\mathfrak g})$. In \cite {MR97k:17014}, the imaginary Verma module for the affine Lie algebra $ \widehat{\mathfrak{sl}(2)}$ is $q$-deformed in such a way that the weight multiplicities, both finite and infinite-dimensional, are preserved. This construction is generalized to the imaginary Verma modules for any affine Lie algebra $\widehat{\mathfrak g}$ of type 1 in \cite{MR1662112}. Furthermore, in \cite {MR97k:17014}, the authors used a special technique involving the Diamond Lemma to construct a PBW type basis for the quantum imaginary Verma module for $U_q(\widehat{\mathfrak{sl}(2)})$ which does not generalize to other affine Lie algebras. In this paper we use a different approach and construct a PBW type basis for the quantum imaginary Verma module for any quantum affine algebra $U_q(\widehat{\mathfrak g})$ of type 1 (see \thmref{mainresult} and \thmref{pbwthm}).

A reader might want to compare our results with those appearing in the work of  Beck, Chari and Pressley \cite{MR1712630}.  In this cited paper their Lemma 1.5 should be compared to our definition of $X_{\beta_r^\pm}$ given below in section 3. 
%On the other hand Definition 1.2 in \cite{MR1712630} is used to give basis for a  triangular decompositions different from the standard one but only by replacing finitely many root vectors in the standard basis (since those functions have finite support). On the other hand in order to obtain our decomposition we need functions with infinite support.   
The difference of our decomposition with that given in the work of Beck, Chari and Pressley has more to do with the breakdown of different quantized Borel subalgebras.  In  \cite{MR1712630}, the authors have a decomposition of the quantized Borel subalgebra coming from the standard positive root system as
$$
U^+\cong U^+(>)\otimes U^+(0)\otimes U^+(<)
$$
where $U^+(>)$ (resp. $ U^+(<)$), is the subalgebra generated by root vectors having a root from the set $\{\alpha+k\delta\,|\,k\geq 0,\enspace\alpha\in \Delta_{0,+}\}$  (resp.$\{-\alpha+k\delta\,|\,k>0,\enspace\alpha\in \Delta_{0,+}\}$) and $ U^+(0)$ is the subalgebra generated by certain root vectors for the set of imaginary positive roots $\{k\delta\,|\,k> 0\}$.    Here $\Delta_0$ denotes the set of roots of $\mathfrak g$ with chosen set of positive/negative roots $\Delta_{0,\pm}$.   Let $U^+_q(S)$ generated by the root vectors coming from the ``natural" or synonamously  imaginary partition of positive roots $S=\{\alpha+k\delta\,|\, k\in\mathbb Z,\alpha \in  \Delta_{0,+}\}\cup \{k\delta\,|\,k\geq 0\}$ for the natural Borel subalgebra.  This subalgebra gives rise to an imaginary Verma module where one induces up using the natural Borel.  If one chooses another partition instead of the natural or standard positive root system one obtains other Verma type modules.  In this paper we focus on quantized imaginary Verma modules which are in a sense the simplest quantized Verma type modules.
We don't decompose $U^+$, rather we decompose the positive part of the quantized enveloping algebra $U^+_q(S)$.
Our main result given below in \thmref{mainresult} is essentially
$$
U^+(S)\cong U^+(>) \otimes \Omega( U^+(<))\otimes U^+(0)\otimes U^0
$$ 
where the algebras on the right are what is defined above as in \cite{MR1712630} and $\Omega$ is an anti-automorphism. 
The braid group action on the root vectors of the quantized enveloping algebra plays a fundamental role in the proof of the our result.   We use them to transform results appearing in Beck (see  \cite{MR1301623} and \cite{MR1298947}), and Damiani's work (see \cite{MR1634087} and \cite{MR1802170}) to our setting.

In  \cite{MR1712630} the authors give an algebraic characterization of the affine canonical basis corresponding to the standard set of positive roots generalizing results of Lusztig \cite{MR1035415} and Kashiwara \cite{MR1115118}.   In future work we will construct an analog of what we call the Kashiwara algebra $\mathcal K_q$ for the imaginary Verma module $M_q(\lambda)$ for the quantum affine algebra $U_q(\widehat{\mathfrak{g}})$ by introducing certain Kashiwara-type operators.
Then we will prove that a certain quotient $\mathcal N_q^-$ of  $U_q(\widehat{\mathfrak{g}})$ is a simple $\mathcal K_q$-module.
This has already been done in the setting of $U_q(\widehat{\mathfrak{sl}(2)})$ (see \cite{MR2642563}).  Our eventual aim is to provide an algebraic characterization of a ``reduced" canonical basis for the quantized imaginary reduced Verma module constructed from the imaginary Borel defined by the set of positive roots $S$.

We would like to thank Ilaria Damiani for making very useful suggestions for this paper. 

%\section{Preliminaries}

\section{ The affine Lie algebra $ \widehat{\mathfrak{g}}$.}

We begin by recalling some basic facts and constructions for the affine
Kac-Moody algebra $\widehat{\mathfrak{g}}$
and its imaginary Verma modules.
See \cite{K} for Kac-Moody algebra terminology and standard notations.

\subsection{} \label{notation} 
Let $I=\{0,\dots, N\}$, $I_0=\{1,2,\dots,N\}$, and $ A=(a_{ij})_{0\leq i,j\leq N}$ be a generalized affine Cartan matrix of type 1 for an untwisted affine Kac-Moody algebra $\widehat{\mathfrak g}$.   Let $D=(d_0,\dots,d_N)$ be a diagonal matrix with relatively prime integer entries such that the matrix $DA$ is symmetric.
%Let $\check P=\oplus_{i=1}^N\check \omega_i$ be the coweight lattice, $\check Q:=\oplus_{i=1}^N\mathbb Zh_i$ the coroot lattice of $\dot{\mathfrak g}$, where $h_i:=\sum_{j=1}^Na_{ij}\check\omega_j$.    One has $\dot Q=\text{Hom}(\check P,\mathbb Z)$ where $\langle \alpha_i,\check\omega_j\rangle =\delta_{ij}$.  Let $\Pi=\{\alpha_1,\dots,\alpha_N\}$ be the set of simple roots and $\check \Pi:=\{\check\alpha_1,\dots,\check\alpha_N\}$ be the set of simple coroots.   Define the reflection $s_i$, $1\leq i\leq N$, on $\check P$ (resp. $\dot Q$) by $s_i(r)=r-\langle \alpha_i,r\rangle h_i$ (resp. $s_i(z)=z-\langle z, h_i\rangle\alpha_i$).  The set of reflections $s_1,\dots,s_N$ generate as subgroup $W_0$ of $\text{Aut}\, \check P$.  Define the root system $R:=W_0\Pi$ (resp. coroot system $\check R=W_0\check \Pi$).  Let $\tilde Q=\mathbb Z\alpha_0 \oplus \dot Q$ denote the affine root lattice and set $\delta:=\alpha_0+\theta$ where $\theta$ is the highest root of $R$.
Then $\widehat{\mathfrak g}$ has the loop space realization
$$
\widehat{\mathfrak g}=  {\mathfrak g} \otimes {\mathbb C}[t,t^{-1}] \oplus {\mathbb C} c \oplus {\mathbb C} d, 
$$
where $\mathfrak g$ is the finite dimensional simple Lie algebra over $\mathbb C$ with Cartan matrix $(a_{ij})_{1\leq i,j\leq N}$, $c$ is central in $\widehat{\mathfrak g}$; $d$ is the degree derivation,
so that $[d,x \otimes  t^n] = n x \otimes  t^n$ for any $x \in { \mathfrak g}$ and
$n \in {\mathbb Z}$, and $[x \otimes  t^n, y \otimes  t^m] = [x,y] \otimes  t^{n+m} +
\delta_{n+m,0}n(x|y)c$ for all $x,y \in  {\mathfrak g}$, $n,m \in {\mathbb Z}$.  

An alternative Chevalley-Serre presentation of $\widehat{\mathfrak g}$ is given
by defining it as the Lie algebra with generators $e_i, f_i, h_i$
 ($i \in   I$) and $d$ subject to the relations
\begin{align}\label{Chevalley-Serre}
[h_i,h_j] &= 0, \qquad [d,h_i]=0,\\
[h_i, e_j] &= a_{ij}e_j, \qquad [d,e_j] = \delta_{0,j}e_j,\\
[h_i, f_j] &= -a_{ij}f_j, \qquad [d, f_j] = -\delta_{0,j}f_j,\\
[e_i, f_j] &= \delta_{ij}h_i, \\
({{\text {ad}}} e_i)^{1-a_{ij}}(e_j)&= 0, \qquad 
({{\text {ad}}} f_i)^{1-a_{ij}}(f_j) = 0, \quad i \neq j.
\end{align}
We
set $\hat{\mathfrak h}$ to be the span of $\{h_0,\dots, h_N,d\}$. 

Let $\Delta_0$ be the set of roots of $\mathfrak g$ with chosen set of positive/negative roots $\Delta_{0,\pm}$.  Let $ Q_0$ be the free abelian group with basis $\alpha_i$, $1\leq i\leq N$ which is the root lattice of $\mathfrak g$. Let $\check Q_0=\sum_i\mathbb Zh_i$ be the coroot lattice of $\mathfrak g$.   The co-weight lattice is defined to be $\check P_0=\text{Hom}(Q_0,\mathbb Z)$ with basis $\omega_i$ defined by $\langle \omega_i,  \alpha_j\rangle =\delta_{i,j}$.  The simple reflections $s_i:\check P_0\to \check P_0$ are defined by $s_i(x)=x-\langle \alpha_i,x\rangle h_i$.  The $s_i$ also act on $Q_0$ by $s_i(y)=y-\langle y,h_i\rangle \alpha_i$.
The Weyl group of $\mathfrak g$ is defined as the subgroup $W_0$ of $\text{Aut}\check P_0$ generated by $s_1,\dots,s_N$.  The affine Weyl group is defined as $W=W_0\ltimes \check Q_0$.  Let $\theta=\sum_{i=1}^Na_i\alpha_i$ be the highest positive root with $a_i$ labels of the extended Dynkin diagram and set $s_0=(s_\theta,-\check \theta)$ where 
$$
s_\theta(\lambda)=\lambda-\langle \lambda,\check\theta\rangle \theta
$$
for all $\lambda \in \mathfrak h^*$.  Note if $\alpha=\sum_ik_i\alpha_i$, then 
$$
\check\alpha=\sum_i\frac{(\alpha_i|\alpha_i)}{(\alpha|\alpha)}k_i\check\alpha_i
$$
(see \cite{K} formula (5.1.1) and Ch. 2 for the notation $(\enspace|\enspace)$.)  Note also $h_i=\check\alpha_i$.

  Then $W$ is generated by $s_0,\dots, s_N$.  Let $\tilde W=W_0\ltimes \check P_0\cong T\ltimes W$ be the generalized affine Weyl group
where $T$ is the group of Dynkin diagram automorphisms.  The length of element $\tilde w=\tau w\in \tilde W$ with $\tau \in T$ and $w\in W$ is defined by $l(\tilde w)=l(w)$.

Let $\Delta$ be the root system of $\widehat{\mathfrak g}$ with positive/negative set of roots $\Delta_{\pm}$and simple roots $\Pi=\{\alpha_0,\dots, \alpha_N\}$.  Define $\delta=\alpha_0+\theta$. Extend the root lattice $Q_0$ of $ \mathfrak g$ to the affine root lattice
$Q:=  Q _0\oplus {\mathbb Z}\delta$, and extend the form $(.|.)$ to $Q$
by setting $(q|\delta)=0$ for all $q \in Q_0$ and $(\delta|\delta)=0$. The generalized affine Weyl group $\tilde W$ acts on $Q$ as an affine transformation group.  In particular if $z\in \check P_0$ and $1\leq i\leq N$, then
$z(\alpha_i)=\alpha_i-\langle z, \alpha_i\rangle \delta$. 
Let $Q_+ = \sum_{i \in I_0} {\mathbb Z}_{\ge 0}\alpha_i \oplus {\mathbb Z}_{\ge 0}\delta$.

The root system $\Delta$ of $\widehat{\mathfrak g}$ is given by
$$
\Delta = \{ \alpha + n \delta\ |\ \alpha \in  \Delta_0, n \in {\mathbb Z}\}
\cup
\{k\delta\ |\ k\in {\mathbb Z}, k \neq 0\}.
$$
The roots of the form $\alpha+n\delta$, $\alpha \in   \Delta, n \in {\mathbb Z}$ are
called real roots, and those of the form $k\delta$,
$k \in {\mathbb Z}, k \neq 0$ are called imaginary roots.  We let
$\Delta^{re}$ and $\Delta^{im}$ denote the sets of real and
imaginary roots, respectively.  The set of positive real roots
of  $\widehat{\mathfrak g}$ is $\Delta_+^{re} =  \Delta_{0,+}
\cup \{\alpha + n\delta \ |\ \alpha \in  \Delta_0, n>0\}$ and the set of positive
imaginary roots is $\Delta_+^{im} = \{k\delta\ |\ k>0\}$.
The set of positive roots of $\widehat{\mathfrak g}$ is
$\Delta_+ = \Delta_+^{re} \cup \Delta_+^{im}$.  Similarly, on the negative
side, we have $\Delta_- = \Delta_-^{re}\cup \Delta_-^{im}$, where
$\Delta_-^{re} = \Delta_{0,-} \cup \{ \alpha + n\delta\ |\ \alpha \in  \Delta_0, n<0\}$
and $\Delta_-^{im} = \{k\delta\ |\ k<0\}$.
The 
weight lattice $P$ of ${\widehat{\mathfrak g}}$ is
$P = \{ \lambda \in {\widehat{\mathfrak h}}^*\ |\ \lambda(h_i) \in {\mathbb Z}, i \in I, \lambda(d) \in {\mathbb Z}\}$.
Let $B$ denote the associated braid group with generators
$T_0, T_1, \dots, T_N$.

%
%Let $\Delta$ denote the root system of $A_1^{(1)}$, and let
%$\{ \alpha_0, \alpha_1\}$ be a basis for $\Delta$.  Let $\delta = \alpha_0 + \alpha_1$,
%the minimal imaginary root.  Then
%$$
%\Delta = \{ \pm \alpha_1 + n\delta\ |\ n \in \mathbb Z\} \cup \{ k\delta\ |\ k \in \mathbb Z
%\setminus \{ 0 \} \}.
%$$

%\subsection{}
%The universal enveloping
%algebra $
%U(A_1^{(1)})$ of $A_1^{(1)}$
%is the associative algebra
%over $\mathbb C$ with 1
%generated by the elements $h_0, h_1, d, e_0, e_1, f_0, f_1$
%with defining relations
%\begin{align*}&[h_0,h_1]=[h_0,d] = [h_1,d]=0, \\
%&h_ie_j-e_jh_i = a_{ij}e_j, \quad h_if_j-f_jh_i=-a_{ij}f_j, \\
%&d e_j-e_j d =\delta_{0,j}e_j, \quad d f_j-f_j d = -\delta_{0,j}f_j, \\
%&e_if_j-f_je_i = \delta_{ij}h_i, \\
%&e_je_i^3-3e_ie_je_i^2+3e_i^2e_je_i-e_i^3e_j = 0 \text{ for } i \neq j, \\
%&f_jf_i^3-3f_if_jf_i^2+3f_i^2f_jf_i-f_i^3f_j = 0 \text{ for } i \neq j.
%\end{align*}
% Corresponding to the loop algebra formulation of $A_1^{(1)}$ is an
%alternative description of
%$U(A_1^{(1)})$ as the associative algebra over $\mathbb C$ with 1 generated
%by the elements $e(k), f(k)$ $(k \in \mathbb Z)$, $h(l)$ $(l \in
%\mathbb Z\setminus \{ 0\})$,
%$c, d, h$, with relations
%\begin{align*}& [c,u]=0 \ \ \text {for all} \ u\in U(A_1^{(1)}), \\
%& [h(k), h(l)]=2k \delta_{k+l,0} c, \\
%& [h,d]=0, \ \ [h, h(k)]=0, \\
%& [d,h(l)]=l h(l), \ \ [d,e(k)]=ke(k), \ \ [d,f(k)]=kf(k),\\
%& [h,e(k)]=2e(k), \ \ [h,f(k)]=-2f(k), \\
%& [h(k), e(l)]=2e(k+l), \ \ [h(k), f(l)]=-2f(k+l), \\
%& [e(k), f(l)]=h(k+l)+k \delta_{k+l,0} c.
%\end{align*}

\subsection{}  \label{partition}
Consider the partition $\D = S \cup -S$ of the root system of $\g$
where $S=\{ \alpha+ n \delta\ |\ \alpha\in \D_{0,+}, n \in \mathbb Z\} \cup
\{k\delta\ |\ k>0\}$. This is a non-standard partition of the root system
$\D$ in the sense that $S$ is not Weyl equivalent to the set
$\D_+$ of positive roots.  There are other non-standard partitions of the root system $\Delta$, but we leave the study of Verma type modules arising from these other partitions to future work.   The reason we stick to the case of the above $S$ for imaginary Verma modules is that they are perhaps the least technically complicated to work with when considering all non-conjugate non-standard partitions of $\Delta$.

\section{The quantum affine algebra $U_q(\widehat{\mathfrak g})$}

\subsection{}  \label{jimbodrinfeld}
The {\it quantum affine algebra}
$U_q(\widehat{\mathfrak g})$ is the $\mathbb C(q^{1/2})$-algebra with 1 generated by
$$ 
E_i, \enspace F_i, \enspace K_\alpha,\enspace \gamma^{\pm 1/2}, \enspace D^{\pm 1} \quad 0\leq i\leq N,\enspace \alpha\in Q,
$$
and defining relations:
\begin{align*}
& DD^{-1}=D^{-1}D=\gamma^{1/2}\gamma^{-1/2} =\gamma^{-1/2}\gamma^{1/2}=1, \\
&K_\alpha K_\beta =K_{\alpha+\beta}, K_0=1, \\
&[\gamma^{\pm 1/2},U_q(\mathfrak g)]=[D,K_i^{\pm1}]=[K_i,K_j]=0, \\
&(\gamma^{\pm 1/2})^2=K_\delta^{\pm 1},\\
& E_iF_j-F_jE_i = \delta_{ij}\frac{K_i-K_i^{-1}}{q_i-q_i^{-1}}, \\
& K_\alpha E_iK_\alpha^{-1}=q^{(\alpha|\alpha_i)}E_i, \ \ K_\alpha F_i K_\alpha^{-1} =q^{-(\alpha|\alpha_i)}F_i, \\
& DE_iD^{-1}=q^{\delta_{i,0}} E_i,\quad DF_iD^{-1}=q^{-\delta_{i,0}} F_i, \\
& \sum_{s=0}^{1-a_{ij}}(-1)^sE_i^{(1-a_{ij}-s)} E_jE_i^{(s)}=0= \sum_{s=0}^{1-a_{ij}}(-1)^sF_i^{(1-a_{ij}-s)} F_jF_i^{(s)}, \quad i\neq j.
\end{align*}
where
$$
q_i:=q^{d_i},\quad 
[n]_i= \frac{q^n_i-q^{-n}_i}{q_i-q^{-1}_i},\quad [n]_i!:=\prod_{k=1}^n[k]_i
$$
and 
$K_i=K_{\alpha_i}$, $E_i^{(s)}=E_i/[s]_i!$ and $F_i^{(s)}=F_i/[s]_i!$ (see \cite{MR1301623} and \cite{MR954661}). 

The quantum affine algebra
$U_q(\widehat{\mathfrak g})$  is a Hopf algebra with a comultiplication given by
\begin{align}
\Delta(K_i^{\pm 1}) &= K_i^{\pm 1} \otimes K_i^{\pm 1},\label{ki} \\
\Delta(D^{\pm 1})&=D^{\pm 1}\otimes D^{\pm 1},\qquad \Delta(\gamma^{\pm 1/2})=\gamma^{\pm 1/2}\otimes \gamma^{\pm 1/2}\label{dgamma}\\
\Delta(E_i) &= E_i\otimes 1 + K_i\otimes E_i, \label{ei}\\
\Delta(F_i) &= F_i\otimes K_i^{-1} + 1 \otimes F_i,\label{fi} 
\end{align}
and an antipode given by
\begin{gather*}
s(E_i) =-E_iK_i^{-1},  \quad 
s(F_i) = -K_iF_i, \\
s(K_i) = K_i^{-1}, \quad 
s(D)  = D^{-1}, \quad 
s(\gamma^{1/2}) = \gamma^{-1/2}.
\end{gather*}

Let $\Phi:U_q(\widehat{\mathfrak g})\to U_q(\widehat{\mathfrak g})$ be the $\mathbb C $-algebra automorphism defined by
\begin{gather}
\Phi(E_i)=F_i,\enspace \Phi(F_i) =E_i,\enspace \Phi(K_\alpha)=K_{\alpha},\label{automorphism} \\
 \Phi(D)=D,\enspace \Phi(\gamma^{\pm1/2})=\gamma^{\pm 1/2},\enspace \Phi(q^{\pm 1/2})=q^{\mp 1/2},\notag
\end{gather}
and let $\Omega:U_q(\widehat{\mathfrak g})\to U_q(\widehat{\mathfrak g})$ be the $\mathbb C $-algebra anti-automorphism defined by
\begin{gather}\label{antiautomophism}
\Omega(E_i)=F_i,\enspace \Omega(F_i)=E_i,\enspace \Omega(K_\alpha)=K_{-\alpha},\\ \Omega(D)=D^{-1},\enspace \Omega(\gamma^{\pm1/2})=\gamma^{\mp 1/2},\enspace \Omega(q^{\pm 1/2})=q^{\mp1/2},\notag
\end{gather}
(see \cite[Section 1]{MR1301623}).
%\color{blue}
%The anti-automorphism $\Omega$ should not be confused with the operators $\Omega_{\psi_i}$ and $\Omega_{\phi_i}$ given below in \eqnref{definingomegapsi}.
%
%\color{black}

\subsection{}\label{secondrealization}  There is an alternative realization for $U_q(\widehat{\mathfrak{g}})$,
due to Drinfeld
\cite{MR802128}, which we shall also need.   We will use the formulation due to J. Beck \cite{MR1301623}.   Let
$U_q(\widehat{\mathfrak g})$ be the associative algebra with $1$ over $\mathbb C(q^{1/2})$-
generated by
$$  
x_{ir}^{\pm 1},\enspace h_{is}, \enspace K_i^{\pm 1}, \enspace \gamma^{\pm 1/2},D^{\pm 1} \enspace 1\leq i\leq N,r,s\in\mathbb Z, s\neq 0,
$$
with defining relations:
\begin{align}
 DD^{-1}&=D^{-1}D=K_iK_i^{-1}=K_i^{-1}K_i=\gamma^{1/2}\gamma^{-1/2}=\gamma^{-1/2}\gamma^{1/2}=1,\label{drinfeldfirst} \\
[\gamma^{\pm 1/2},U_q(\mathfrak g)]&=[D,K_i^{\pm 1}]=[K_i,K_j]=[K_i,h_{jk}]=0, \\
Dh_{ir}D^{-1}&=q^rh_{ir},\quad Dx_{ir}^{\pm}D^{-1}=q^rx_{ir}^{\pm},\\
K_ix_{jr}^{\pm}K_i^{-1} &= q_i^{\pm  (\alpha_i|\alpha_j)}x_{jr}^{\pm},   \\  
[h_{ik},h_{jl}]&=\delta_{k,-l} \frac{1}{k}[ka_{ij}]_i\frac{\gamma^k-\gamma^{-k}}{q_j-q_j^{-1}},\label{hs} \\
[h_{ik},x^{\pm}_{jl}]&= \pm \frac{1}{k}[ka_{ij}]_i\gamma^{\mp |k|/2}x^{\pm}_{j,k+l}, \label{axcommutator}  \\
    x^{\pm}_{i,k+1}x^{\pm}_{jl} &- q^{\pm (\alpha_i|\alpha_j)}
    %%%%%%%%%%%%%%%%%%%%%
x^{\pm}_{jl}x^{\pm}_{i,k+1}\label{Serre}   \\
&= q^{\pm (\alpha_i|\alpha_j)}x^{\pm}_{ik}x^{\pm}_{j,l+1}
    - x^{\pm}_{j,l+1}x^{\pm}_{ik},\notag \\
[x^+_{ik},x^-_{jl}]&=\delta_{ij}
    \frac{1}{q_i-q^{-1}_i}\left( \gamma^{\frac{k-l}{2}}\psi_{i,k+l} -
    \gamma^{\frac{l-k}{2}}\phi_{i,k+l}\right), \label{xcommutator}   \\
\text{where  }
\sum_{k=0}^{\infty}\psi_{ik}z^{k} &= K_i \exp\left(
(q_i-q^{-1}_i)\sum_{l>0}  h_{il}z^{l}\right), \text{ and }\notag\\
\sum_{k=0}^{\infty}
\phi_{i,-k}z^{-k}&= K^{-1}_i \exp\left( - (q_i-q^{-1}_i)\sum_{l>0} 
h_{i,-l}z^{-l}\right).\label{phidef}\\
\text{For }i\neq j,\enspace n:=1-a_{ij} \notag  \\
\text{Sym}_{k_1,k_2,\dots,k_n}&\sum_{r=0}^{n}(-1)^r \genfrac{[}{]}{0pt}{}{n}{r} x_{ik_1}^\pm \cdots x_{ik_r}^\pm x_{jl}^\pm x_{ik_{r+1}}^\pm \cdots x_{ik_s}^\pm=0.\label{drinfeldlast}
\end{align}
 Note that Beck's paper \cite{MR1301623} on page 565 has a typo in it where he has $\phi_{i,k}z^k$ instead of $\phi_{i,-k}z^{-k}$.

In the above last relation $\text{Sym}$ means symmetrization with respect to the indices $k_1,\dots, k_n$.  Also in Drinfeld's notation one has $e^{hc/2}=\gamma$ and $e^{h/2}=q$.

The algebras given above and in \secref{jimbodrinfeld} are isomorphic \cite{MR802128}. 
%The action of the
%isomorph ism,
%which we shall call the {\it Drinfeld Isomorphism,} on the generators of
%$U_q(A_1^{(1)})$ is:
%\begin{align*}e_0 &\mapsto x^-(1)K^{-1}, \ \ f_0 \mapsto Kx^+(-1), \\
%e_1 &\mapsto x^+(0), \ \ f_1 \mapsto x^-(0), \\
%K_0 &\mapsto \gamma K^{-1}, \ \ K_1 \mapsto K, \ \ D \mapsto D.
%\end{align*}
If one uses the formal sums
\begin{equation}
\phi_i(u)=\sum_{p\in\mathbb Z} \phi_{ip}u^{-p},\enspace \psi_i(u)=\sum_{p\in\mathbb Z}\psi_{ip}u^{-p},\enspace
x_i^{\pm }(u)=\sum_{p\in\mathbb Z} x_{ip}^\pm u^{-p}
\end{equation}
Drinfeld's relations \eqnref{hs}-\eqnref{xcommutator} can be written as
\begin{gather}
[\phi_i(u),\phi_j(v)]=0=[\psi_i(u),\psi_j(v)] \\
\phi_i(u)\psi_j(v)\phi_i(u)^{-1}\psi_j(v)^{-1}=g_{ij}(uv^{-1}\gamma^{-1})/g_{ij}(uv^{-1}\gamma)  \\ 
\phi_i(u)x^\pm_j (v)\phi_i(u)^{-1}=g_{ij}(uv^{-1}\gamma^{\mp 1/2})^{\pm 1}x_j^\pm (v)\label{phix} \\
\psi_i(u)x_j^\pm (v)\psi_i(u)^{-1}=g_{ji}(vu^{-1}\gamma^{\mp 1/2})^{\mp 1}x_j^\pm (v)\label{psix} \\
(u-q^{\pm(\alpha_i|\alpha_j)} v)x^\pm_i (u)x^\pm_j (v)=(q^{\pm (\alpha_i|\alpha_j)}u-v)x_j^\pm(v)x_i^\pm(u) \\
[x_i^+(u),x_j^-(v)]=\delta_{ij}(q_i-q^{-1}_i)^{-1}(\delta(u/v\gamma)\psi_i(v\gamma^{1/2})-\delta(u\gamma/v)\phi_i(u\gamma^{1/2}))\label{xx}
\end{gather}
where $g_{ij}(t)=g_{ij,q}(t)$ is the Taylor series at $t=0$ of the function $(q^{(\alpha_i|\alpha_j)}t-1)/(t-q^{(\alpha_i|\alpha_j)})$ and $\delta(z)=\sum_{k\in\mathbb Z}z^{k}$ is the formal Dirac delta function.
%\begin{rem}
%Writing $g_{ij}(t)=g_{ij,q}(t)=\sum_{p\geq 0}g_{ij}(p)t^p$ we have
%$$g_{ij}(0)=q^{-(\alpha_i|\alpha_j)}, \quad g_{ij}(p)=(1-q^{2(\alpha_i|\alpha_j)})q^{-(p+1)(\alpha_i|\alpha_j)}, \quad p>0.$$
%Note that $g_{ij,q}(t)^{-1}=g_{ij,q^{-1}}(t)$ and $g_{ij,q}(t)=g_{ji,q}(t)$ as $(\cdot|\cdot)$ is symmetric.
%\end{rem}

%We will need the following identity later in \lemref{le-length2}:
%\begin{lem}
%\begin{equation}
%\label{quteformula}
%\text{\rm exp}\,\left((q_i-q^{-1}_i)
%\sum_{k=1}^\infty \frac{-[ka_{ij}]_i}{k}z^{-k}\right)=1+(1-q^{2a_{ij}}_i)\sum_{r=1}^\infty \left(zq^{a_{ij}}_i\right)^{-r}=q^{a_{ij}}_ig_{ij}(1/z)
%\end{equation}
%\end{lem}
%\begin{proof}
%\begin{align*}
%\text{\rm exp}\,\left((q_i-q_i^{-1})
%\sum_{k=1}^\infty \frac{-[ka_{ij}]_i}{k}z^{-k}\right)&=\text{\rm exp}\,\left(\sum_{k=1}^\infty \frac{1}{k}(zq_i^{a_{ij}})^{-k}-\sum_{k=1}^\infty \frac{1}{k}\left(\frac{z}{q^{a_{ij}}_i}\right)^{-k}\right)  \\
%&=\text{\rm exp}\,\left(\sum_{k=1}^\infty \frac{1}{k}(zq_i^{a_{ij}})^{-k}\right)\text{\rm exp}\,\left(-\sum_{k=1}^\infty \frac{1}{k}\left(\frac{z}{q^{a_{ij}}_i}\right)^{-k}\right)  \\
%&=\left(\frac{1}{1-\frac{1}{zq^{a_{ij}}_i}}\right)\left(1-\frac{q^{a_{ij}}_i}{z}\right)\\
%&=\sum_{k=0}^\infty \left(\frac{1}{zq^{a_{ij}}_i}\right)^k-q^{2a_{ij}}_i\sum_{k=1}^\infty \left(\frac{1}{zq^{a_{ij}}_i}\right)^{k}.
%\end{align*}
%\end{proof}
\subsection{}\label{drinfeldgen}

Let $U_q^+=U_q^+(\g)$ (resp. $U_q^-=U_q^-(\g)$) be the subalgebra of
$U_q(\g)$ generated by $E_i$ (resp. $F_i$), $i \in I$, and
let $U_q^0=U_q^0(\g)$ denote the subalgebra generated by
$K_i^{\pm 1}$ ($i \in I$) and $D^{\pm 1}$.

%
%Beck, \cite{MR1301623, MR1298947} has introduced a total ordering of the root
%system leading to  PBW bases for $\widehat{\mathfrak g}$ and its quantum analog,
%$U_q(\widehat{\mathfrak g})$.  We state the construction here, partially following the
%more abstract notation .  For a related PBW construction, see
%\cite{MR1230442}.

Beck in \cite{MR1301623} and \cite{MR1298947} has given a total ordering of the
root system $\D$ and a PBW like basis for $U_q(\g)$.
Below we follow the construction developed by Damiani \cite{MR1634087}, Gavarini \cite{MR1672011} and \cite{MR1317228} and
let $E_{\beta}$ denote the root vectors for each
$\beta \in \Delta_+$ counting with multiplicity for the
imaginary roots.  One defines $F_\beta=E_{-\beta}:=\Omega(E_\beta)$ for $\beta\in\Delta_+$ (refer to \eqnref{antiautomophism}).

For any affine Lie algebra $\widehat{\mathfrak g}$, there exists a map $\pi :\mathbb Z \to I$
such that, if we define
$$
\beta_k =
\begin{cases} 
&s_{\pi(0)}s_{\pi(-1)}\cdots s_{\pi(k+1)}(\alpha_{\pi(k)})
\qquad \text{ for all } k < 0, \\
&\alpha_{\pi(0)}\hskip 130pt k=0, \\
& \alpha_{\pi(1)}\hskip 130pt k=1, \\
&s_{\pi(1)}s_{\pi(2)} \cdots s_{\pi(k-1)}(\alpha_{\pi(k)})
\qquad  \text{ for all } k> 1,
\end{cases}
$$
then the map $\pi':\mathbb Z \mapsto \Delta_+^{re}$ given by $\pi'(k)=\beta_k$
is a bijection. Note that the map $\pi$, and hence the total ordering, is not unique.
We fix $\pi$ so that
$ \{ \beta_k\ |\ k \le 0\} =
\{ \alpha + n\delta \ |\ \alpha \in  \Delta_{0,+}, n \ge 0\}$ and
$\{ \beta_k\ |\ k \ge 1\} =
\{ -\alpha + n\delta\ |\ \alpha \in  \Delta_{0,+}, n>0\}$.  
One also defines the set of imaginary roots with multiplicity as
%subsets of $\Delta$ by
\begin{gather*}
%\Delta_+(k):=\begin{cases} \{\beta_r\,|\,1\leq r< k\},\quad \text{ if }k\geq 1, \\
% \{\beta_r\,|\,0\leq r> k\},\quad \text{ if }k\leq 0; \\
%\end{cases}  \\
%\Delta_+(\infty):=\left\{\beta_r\,|\, r\geq 1\right\},\quad \Delta_+(-\infty):=\left\{\beta_r\,|\, r\leq 0\right\} , \\
\Delta_+(\text{im}):=\Delta_+^{\text{im}}\times  I_0,
%\Delta_+(\tilde\infty):=\Delta_+(\infty)\cup \Delta_+(\text{im}),\quad \Delta_+(-\tilde\infty):=\Delta_+(-\infty)\cup \Delta_+(\text{im}).
\end{gather*}
where we recall $I_0=\{1,...,N\}$.
%\color{red} Is the above needed for what follows?\color{black}

It will be convenient for us to invert Beck's original ordering
of the positive roots (see \cite[\S 1.4.1]{MR1317228}).   Let
\begin{equation}
\beta_0 > \beta_{-1} > \beta_{-2}> \dots >\delta >2\delta>  \dots
>\beta_2 >\beta_1,
\end{equation}
(\cite[\S 2.1]{MR1672011} for this ordering).
We define $-\alpha < -\beta$ iff $\alpha > \beta$ for all
positive roots $\alpha, \beta$, so we obtain a corresponding ordering
on $\Delta_-$.

The following elementary observation on the ordering will play
a crucial role later.  Write $A<B$ for two sets $A$ and $B$ if
$x < y$ for all $x \in A$ and $y \in B$.  Then Beck's total
ordering of the positive roots can be divided into three sets:
$$
\{ \alpha + n\delta\ |\ \alpha \in  \Delta_{0,+}, n \ge 0\} >
\{k\delta \ |\ k>0\} > \{-\alpha+k\delta\ |\ \alpha \in   \Delta_{0,+}, k>0\}.
$$
Similarly, for the negative roots, we have,
$$
\{ -\alpha - n\delta\ |\ \alpha \in \Delta_{0,+}, n \ge 0\} <
\{-k\delta \ |\ k>0\} < \{\alpha-k\delta\ |\ \alpha \in   \Delta_{0,+}, k>0\}.
$$

The action of the braid group generators $T_i$ on the generators
of the quantum group $U_q(\widehat{\mathfrak g})$ is given by the following.
\begin{align*}
T_i(E_i) &= -F_i K_i, \qquad T_i(F_i) = -K_i^{-1}E_i, \\
T_i(E_j) &= \sum_{r=0}^{-a_{ij}} (-1)^{r-a_{ij}}
q_i^{-r}
E_i^{(-a_{ij}-r)}E_jE_i^{(r)}, \qquad \text{if } i \neq j,\\
T_i(F_j) &=  \sum_{r=0}^{-a_{ij}} (-1)^{r-a_{ij}}q_i^r
F_i^{(r)}F_jF_i^{(-a_{ij}-r)}, \qquad \text{if } i \neq j,\\
T_i(K_j) &= K_jK_i^{-a_{ij}}, \qquad
T_i(K_j^{-1}) = K_j^{-1}K_i^{a_{ij}}, \\
T_i(D) &= DK_i^{-\delta_{i,0}}, \qquad
T_i(D^{-1}) = D^{-1}K_i^{\delta_{i,0}}.
\end{align*}

For each $\beta_k \in \Delta_+^{re}$, define the root vector
$E_{\beta_k}$ in $U_q(\g)$ by
\begin{align}
E_{\beta_k} =
\begin{cases}
&T^{-1}_{\pi(0)}T^{-1}_{\pi(-1)}\cdots T^{-1}_{\pi(k+1)}(E_{\pi(k)})
\quad \text{ for all } k < 0, \\
&E_{\pi(0)}  \hskip 120ptk=0,\\
&E_{\pi(1)} \hskip 120pt k = 1, \\
&T_{\pi(1)}T_{\pi(2)} \cdots T_{\pi(k-1)}(E_{\pi(k)})
\qquad  \text{ for all } k>1.
\end{cases}\label{realrootvector}
\end{align}

The following result is due to Iwahori, Matsumoto and Tits (see \cite{MR1301623}, Section 2).
\begin{prop}
Suppose $w\in\tilde W$ and $w=\tau s_{i_1}\cdots s_{i_n}$ is a reduced decomposition in terms of simple reflections and $\tau\in T$.   Then $T_w:=\tau T_{i_1}\cdots T_{i_n}$ does not depend on the reduced decomposition of $w$ chosen, but rather only on $w$. 
\end{prop}

 Orient the Dynkin diagram of $\mathfrak g$ by defining a map $o:V\to \{\pm 1\}$ so that for adjacent vertices $i$ and $j$ one has $o(i)=-o(j)$.   Beck defines  $\widehat{T}_{\omega_i}=o(i)T_{\omega_i}$ 
and obtains (\cite[Section 4]{MR1301623}) for $i\in I_0$ and $k\in\mathbb Z$, 
\begin{align*}
x_{ik}^-:&=\widehat{T}_{\omega_i}^k(F_i), \enspace x_{ik}^+:=\widehat{T}_{\omega_i}^{-k}(E_i).
\end{align*}

Fix $i\in I_0$ and $k\geq 0$. The proposition above in the particular case of the reduced decomposition of $\omega_i=\tau s_{i_1}\cdots s_{i_r}\in \check P_0\subset \tilde W$ where $\tau$ is a diagram automorphism and the $s_i$ are simple reflections, gives
$$
x_{ik}^+=\widehat{T}_{\omega_i}^{-k}(E_i)=o(i)^k( \tau T_{i_1}\cdots T_{i_r})^{-k}(E_i)=o(i)^k  T_{j_1}\cdots T_{j_m}\tau^{-k}(E_i),
$$
for some $j_t\in I$. 

Fixing still $i\in I_0$ and $k\geq 0$, choose now $w_{\alpha_i+k\delta}\in \tilde W$, and $j\in I$, such that $w_{\alpha_i+k\delta}(\alpha_j)=\alpha_i+k\delta$.  Writing $w_{\alpha_i+k\delta}=s_{l_1}\cdots s_{l_p}$ as a reduced decomposition of simple reflections,  Beck defines 
\begin{align*}
E_{\alpha_i+k\delta}:&=T_{w_{\alpha_i+k\delta}}(E_j)=T_{l_1}\cdots T_{l_p}(E_j),
\end{align*}
which according to Lusztig is independent of the choice of $w_{\alpha_i+k\delta}$, its reduced decomposition and $j\in I$.  In particular we can choose $j=\tau^{-k}(i)$ and $w=s_{j_1}\cdots s_{j_m}$, so that $s_{j_1}\cdots s_{j_m}(\alpha_j)=s_{j_1}\cdots s_{j_m}(\alpha_{\tau^{-k}(i)})=\alpha_i+k\delta$.  Then 
\begin{equation}
E_{\alpha_i+k\delta}=T_{j_1}\cdots T_{j_m}(E_{\tau^{-k}(i)})=o(i)^k x^+_{ik}.\label{rootvector1}
\end{equation}
Now one defines
\begin{align}\label{rootvector2}
F_{\alpha_i+k\delta}%=E_{-\alpha_i-k\delta}
&=\Omega(E_{\alpha_i+k\delta})=o(i)^k  \Omega(x^+_{ik})=o(i)^k  \Omega(\widehat{T}_{\omega_i}^{-k}(E_i))  \\
&=o(i)^k  \widehat{T}_{\omega_i}^{-k}(\Omega(E_i))=o(i)^k  \widehat{T}_{\omega_i}^{-k}(F_i) =o(i)^k  x_{i,-k}^-,\notag
\end{align}
as $T_j\Omega=\Omega T_j$ and $T_\tau \Omega=\Omega T_\tau$.

If $k<0$ and $i\in I_0$, then $-\alpha_i-k\delta \in \Delta_+^{\text{re}}$, so that $-\alpha_i-k\delta=\beta_l=s_{\pi(1)}\cdots s_{\pi(l-1)}(\alpha_{\pi(l)})$ for $l> 1$ and $-\alpha_i-k\delta=\beta_l= \alpha_{\pi(1)}$ if $l=1$. Then for $l>1$,
\begin{align}
E_{-\alpha_i-k\delta}&=E_{\beta_l}=T_{\omega_i}^{-k}T_i^{-1}(E_i)=-T_{\omega_i}^{-k}(K_i^{-1}F_i) \label{rootvector3}
\\
&=-o(i)^kT_{\omega_i}^{-k}(K_i^{-1})x_{i,-k}^-=-o(i)^k K_i^{-1}\gamma^{-k}x_{i,-k}^-\notag
\end{align}
as $\omega_i(-\alpha_i)=-\alpha_i+\delta$ (see \secref{notation}) so that $\omega_i^{-k}s_i(\alpha_i)=\omega_i^{-k}(-\alpha_i)=-\alpha_i-k\delta$ and $T_{\omega_i}(K_i^{-1})=K_{-\alpha_i+\delta}$.  Now 
\begin{align}\label{rootvector4}
F_{-\alpha_i-k\delta}=\Omega(E_{-\alpha_i-k\delta})
&=-o(i)^k \Omega(K_i^{-1}\gamma^{-k}x_{i,-k}^-)\\ 
&=-o(i)^k K_i \gamma^{k} x_{i,k}^+.\notag
\end{align}

He also defines for $k>0$
\begin{align*} 
\psi_{ik}&=(q_i-q_i^{-1})\gamma^{k/2}[E_i,\hat T_{\omega_i}^k(F_i)] =(q_i-q_i^{-1})\gamma^{k/2}[E_i, x_{ik}^-], \\
\phi_{i,-k}&=(q_i-q_i^{-1})\gamma^{-k/2}[F_i,\hat T_{\omega_i}^kE_i]=(q_i-q_i^{-1})\gamma^{-k/2}[F_i, x_{i,-k}^+],
\end{align*}
$\psi_{i,0}:=K_i$, $\phi_{i,0}:=K_i^{-1}$, 
and for any $\tau \in T$, 
\begin{equation}
T_\tau(E_i):=E_{\tau(i)}, \quad 
T_\tau(F_i):=F_{\tau(i)},\quad T_\tau(K_i):=K_{\tau(i)}\label{tau}.
\end{equation}
One writes $\tau$ for $T_\tau$.
Note also that $\tau s_i\tau^{-1}=s_{\tau(i)}$ for all $0\leq i\leq n$.  

Each real root space is 1-dimensional, but each imaginary root
space is $N$-dimensional.  Hence, for each positive imaginary
root $k\delta$ ($k>0$) we define $N$ imaginary root vectors,
$E_{k \delta}^{(i)}$ ($i\in  I_0$) by
\begin{align}
\exp\left(
(q_i-q^{-1}_i)\sum_{k=1}^{\infty}E_{k \delta}^{(i)}z^k\right) &=
1 + (q_i-q^{-1}_i)\sum_{k=1}^{\infty} K_i^{-1}[E_i, x_{i,k}^-]z^k\label{imaginaryrootvectordef}
 \\
&=1 +  \sum_{k=1}^{\infty} K_i^{-1} \psi_{ik}\left(\gamma^{-1/2}z\right)^k  \notag  \\
&=  \exp\left( (q_i-q^{-1}_i)\sum_{l>0}  h_{il}\gamma^{-l/2}z^{l}\right). \notag 
 \end{align}
 So $E_{k\delta}^{(i)}=h_{ik}\gamma^{-k/2}$ for all $k>  0$.  For $k<0$ we also define $E_{k\delta}^{(i)}:=\Omega(E_{-k\delta}^{(i)})=h_{ik}\gamma^{k/2}$. Our definition of $E_{k\delta}^{(i)}$ is the same as  \cite{{MR1634087}}, Definition 7.  In particular
 \begin{align*}
  K_i^{-1}[E_i, x_{i,k}^-]&= -o(i)^k \gamma^{-k}K_i^{-1}\left[E_i, K_i E_{-\alpha_i+k\delta} \right]  \\
  &= -o(i)^k \gamma^{-k}K_i^{-1}\left(E_i K_i E_{-\alpha_i+k\delta}-K_i E_{-\alpha_i+k\delta}E_i \right) \\
 &= -o(i)^k \gamma^{-k}\left(K_i^{-1}E_i K_i E_{-\alpha_i+k\delta}- E_{-\alpha_i+k\delta}E_i \right) \\
 &= -o(i)^k \gamma^{-k}\left(q_i^{-2}E_i  E_{-\alpha_i+k\delta}- E_{-\alpha_i+k\delta}E_i \right) .
 \end{align*}
% \color{red} 
% Is the $E_{k\delta}^{(i)}$ the same in Beck's as in Ilaria's?
% \color{black}

 Using these sets and a symbol $*\in \{\pm \infty,\pm\tilde \infty,\text{im}\}$, one defines $U_q^+(*)$ as the subalgebra of $U_q$ generated by $\{E_\alpha\,|\,\alpha\in \Delta_+(*)\}$ and 
\begin{equation}\label{damianiu}
U_q^{\geq 0}(*):=U^+_q(*)U_q^0,\quad U_q^-(*):=\Omega(U^+_q(*)),\quad U_q^{\leq 0}(*):=\Omega(U_+^{\geq 0}(*)).
\end{equation}

Recall that the $R$-``matrices" are defined having values in $U_q(\widehat{\mathfrak g})\widehat{\otimes} U_q(\widehat{\mathfrak g})$  (see \cite{MR94m:17016} for the definition of $U_q(\widehat{\mathfrak g})\widehat{\otimes} U_q(\widehat{\mathfrak g})$ and  \cite{MR1301623}, Section 5) for $1\leq i\leq N$ by 
\begin{align}
R_i&=\sum_{n\geq 0}(-1)^nq_i^{\frac{-n(n-1)}{2}}(q_i-q_i^{-1})^n[n]_i!T_i(F_i^{(n)})\otimes T_i(E_i^{(n)}),  \label{ri}\\ 
&=\sum_{n\geq 0} \frac{(q_i^{-1}-q_i)^nq_i^{\frac{-n(5n-1)}{2}}}{[n]_i!}E_i^{n}K_i^{-n}\otimes F_i^{n}K_i^n, \notag\\ \notag  \\
\bar{R}_i&=T_i^{-1}\otimes T_i^{-1}\circ R_i^{-1}=\sum_{n\geq 0}q_i^{\frac{n(n-1)}{2}}(q_i-q_i^{-1})^n[n]_i!F_i^{(n)}\otimes E_i^{(n)}.\label{barri}
\end{align}
These operators have inverses
\begin{align*}
R_i^{-1}
&=\sum_{n\geq 0} \frac{(q_i-q_i^{-1})^nq_i^{\frac{-n(3n+1)}{2}}}{[n]_i!}E_i^{n}K_i^{-n}\otimes F_i^{n}K_i^n \\  \\
\bar{R}_i^{-1}&=\sum_{n\geq 0}\frac{q_i^{\frac{-n(n-1)}{2}}(q_i^{-1}-q_i)^n}{[n]_i!}F_i^{n}\otimes E_i^{n}
\end{align*}

%One then has the following result :
%\begin{prop}[\cite{MR94m:17016}, Prop. 37.3.1]\label{lusztigdiagt} For $u\in U_q$ one has 
%\begin{enumerate}
%\item 
%$$
%\Delta(T_i(u))=R_i^{-1}(S_i \circ \Delta)(u)R_i,
%$$\label{deltat}
%\item 
%$$
%\Delta(T_i^{-1}(u))=\bar{R}_i^{-1}\left(S_i^{-1} \circ \Delta\right)(u)\bar R_i,
%$$
%where $S_i:=T_i\otimes T_i$.
%
%\end{enumerate}
%\end{prop}

Suppose $w\in \tilde W$ and $\tau s_{i_1}\cdots s_{i_r}$ is a reduced presentation for $w$ where $\tau$ is defined as in \eqnref{tau}.   Beck defines the following ``$R$-matrices":
\begin{align}
R_w&=\tau (S_{i_1}S_{i_2}\cdots S_{i_{r-1}}(R_{i_r})\cdots S_{i_1}(R_{i_2})R_{i_1}),\label{rw}\\ 
\bar R_w&=\tau (S_{i_r}^{-1}S_{i_{r-1}}^{-1}\cdots S_{i_{2}}^{-1}(R_{i_1})\cdots S_{i_r}^{-1}(\bar R_{i_{r-1}})\bar R_{i_r}).  \label{barrw}
\end{align}
%\begin{prop}[\cite{MR1301623}, Lemma 5.2 and Prop. 5.3] Let $1\leq i\in N$, $k\geq 0$. 
%\begin{enumerate}[(a).]
%\item 
%For $w\in W$, the elements $R_w,\bar R_w$ are well defined elements of $U_q(\widehat{\mathfrak g})\widehat{\otimes} U_q(\widehat{\mathfrak g})$. 
%\item  If $w=k\omega_i$ and $k>0$, then
%$$
%\Delta(x_{ik}^-)=R_w^{-1}(x_{ik}^-\otimes K_{-\alpha_i+k\delta}+1\otimes x_{ik}^-)R_w ,
%$$
%\item If $w=k\omega_i$ and $k>0$, then
%$$
%\Delta(x_{i,-k}^-)=\bar{R}_w^{-1}(x_{i,-k}^-\otimes K_{-\alpha_i-k\delta}^{-1}+1\otimes x_{i,-k}^-)\bar R_w .
%$$
%\end{enumerate}
%\end{prop}

\color{black}
Using the root partition $ S = \{\alpha + k\delta  \ |\alpha\in \Delta_{0,+},\ k \in \mathbb Z \} \cup
\{ l\delta \ |\ l \in \mathbb Z_{>0}\}$ from Section 2.3, we define:
 \vskip 5pt
$U_q^+(S)$ to be the subalgebra of $U_q(\widehat{\mathfrak g})$ generated by $x^+_{i,k}$,  $(1\leq i\leq N, k\in \mathbb Z)$ and
$h_{i,l}$ $(1\leq i\leq N,\ l>0)$;
 \vskip 5pt
$U_q^-(S)$ to be the subalgebra of $U_q(\widehat{\mathfrak g})$ generated by $x^-_{i,k}$
$(1\leq i\leq N, k\in \mathbb Z)$ and
$h_{i,-l}$ $(1\leq i\leq N,\ l>0)$, and
\vskip 5pt
$U_q^0(S)$ to be the subalgebra of $U_q(\widehat{\mathfrak g})$ generated by $K^{\pm 1}_i$  $(1\leq i\leq N)$,
$\gamma^{\pm 1/2}$, and $D^{\pm 1}$.   Thus $U_q^0(S)=U_q^0(\hat{\mathfrak g})$.

\subsection{}

Let $\omega$ denote the standard $\mathbb C(q^{1/2})$-linear antiautomorphism of
$U_q(\widehat{\mathfrak g})$, and set $E_{-\alpha} = \omega(E_{\alpha})$ for all $\alpha \in \Delta_+$.
Then $U_q$ has a basis of elements of the form $E_-HE_+$,
where $E_{\pm}$ are ordered monomials in the $E_{\alpha}$,
$\alpha\in \Delta_{\pm}$,  and $H$ is a monomial in $K_i^{\pm 1}$, $\gamma^{\pm 1/2}$,
and $D^{\pm 1}$ (which all commute).

Furthermore, this basis is, in Beck's terminology, convex, meaning
that, if $\alpha, \beta \in \Delta_+$ and $ \beta >  \alpha$, then
\begin{equation}\label{convexity}
E_{\beta}E_{\alpha} - q^{(\alpha|\beta)}E_{\alpha}E_{\beta} =
\sum_{\alpha <\gamma_1<\dots <\gamma_r <\beta}c_{\gamma}
E_{\gamma_1}^{a_1}\cdots E_{\gamma_r}^{a_r}
\end{equation}
for some integers $a_1,\dots, a_r$ and scalars
$c_{\gamma}\in \mathbb C[q, q^{-1}]$,
$\gamma = (\gamma_1,\dots,\gamma_r)$ \cite[Proposition 1.7c]{MR1317228}, \cite{MR1120927}, and
similarly for the negative roots.  The above is called the {\it Levendorski and So{\u\i}belman convexity formula}.

Set $\mathbb A= \mathbb C[q^{1/2},q^{-1/2}, \frac{1}{[n]_{q_i}}, i\in I, n>1]$.   We first begin with a slightly different $\mathbb A$-form than in \cite{MR1662112}. Namely we define this algebra $U_{\mathbb A}=U_{\mathbb A}(\widehat{\mathfrak g})$ to be the $\mathbb A$-subalgebra
of $U_q(\widehat{\mathfrak g})$ with 1 generated by the elements 
$$  
x_{ir}^{\pm 1},\enspace h_{is}, \enspace K_i^{\pm 1}, \enspace \gamma^{\pm 1/2},D^{\pm 1} ,  \Lnum{K_i}{s}{n},\Lnum{D}{s}{n},\Lnum{\gamma}{s}{1}, \Lnum{\gamma\psi_i}{k,l}{1}
$$
for $1\leq i\leq N,r,s\in\mathbb Z,s\neq 0$ where following \cite{MR954661}, for each $i \in I$, $s \in \mathbb Z$ and $n \in  \mathbb Z_+$,
we define the {\it Lusztig elements} in $U_q(\g)$:
\begin{align}
\Lnum{\gamma}{s}{1}_i&= \frac{\gamma^{s}-\gamma^{-s}}{q_i-q^{-1}_i},\\
\Lnum{\gamma\psi_i}{k,l}{1}&= \frac{\gamma^{\frac{k-l}{2}}\psi_{i,k+l} -
    \gamma^{\frac{l-k}{2}}\phi_{i,k+l}}{q_i-q^{-1}_i} \\ 
\Lnum{K_i}{s}{n} &=
\prod_{r=1}^n
\frac{K_iq_i^{s-r+1} - K_i^{-1}q_i^{-(s-r+1)}}{q_i^r-q_i^{-r}}, \quad \text{and}\\
\Lnum{D}{s}{n} &= \prod_{r=1}^n
\frac{Dq_0^{s-r+1} - D^{-1}q_0^{-(s-r+1)}}{q_0^r-q_0^{-r}}.
\end{align}
where $q_0=q^{d_0}$. 
  This $\mathbb A$-form can be shown to be the same as that in  \cite{MR1662112} with the exception that we have added the generators $\gamma^{\pm 1/2}$, $\Lnum{\gamma}{s}{1}_i$ and $\Lnum{\gamma\psi_i}{k,l}{1}$. 
  Let $U_{\mathbb A}^+$ (resp. $U_{\mathbb A}^-$) denote the subalgebra of
$U_{\mathbb A}$ generated by the $x_{ik}^+$, $h_{il}$, where $k\in\mathbb Z$, $l\in\mathbb Z_{>0}$, $1\leq i\leq N$ (resp. $x_{ik}^-$, $h_{il}$, where $k\in\mathbb Z$, $l\in-\mathbb Z_{\geq 0}\backslash\{0\}$, $1\leq i\leq N$ ), $i \in I$, and
let $U_{\mathbb A}^0$ denote the subalgebra of $U_{\mathbb A}$ generated
by the elements $\gamma^{\pm 1/2},K_i^{\pm 1}, \Lnum{K_i}{s}{n}$,
$ D^{\pm 1}, \Lnum{D}{s}{n}$,
$\Lnum{\gamma}{s}{1}_i$, $ \Lnum{\gamma\psi_i}{k,-k}{1}$.
Note if $k+l>0$ (resp. $k+l<0$), then $ \Lnum{\gamma\psi_i}{k,l}{1}\in U_{\mathbb A}^+$ (resp. $ \Lnum{\gamma\psi_i}{k,l}{1}\in U_{\mathbb A}^-$).
\begin{lem}  \cite[Prop. 3.10, Lemma 3.15]{MR1301623}\label{psix2} Set $a=q_i^2\gamma^{1/2}$, $b=q_i^2\gamma^{-1/2}$, $c=-q_i^{a_{ij}}\gamma^{1/2}$, $d=-q_i^{a_{ij}}\gamma^{-1/2}$ for $i\neq j$, $r>0$, $m\in\mathbb Z$.  Then
\begin{align*}
[\psi_{ir},x_{im}^-]&=-\gamma^{1/2}[2]_i\left(\sum_{k=1}^{r-1}a^{k-1}(q_i-q_i^{-1})\psi_{i,r-k}x_{i,m+k}^-+a^{r-1}x_{i,m+r}^-\right)  \\
[\psi_{ir},x_{im}^+]&=\gamma^{-1/2}[2]_i\left(\sum_{k=1}^{r-1}b^{k-1}(q_i-q_i^{-1})x_{i,m+k}^+\psi_{i,r-k}+b^{r-1}x_{i,m+r}^+\right)  \\ 
[\psi_{ir},x_{jm}^-]&=\gamma^{1/2}[a_{ij}]_i\left(\sum_{k=1}^{r-1}c^{k-1}(q_i-q_i^{-1})\psi_{i,r-k}x_{j,m+k}^-+a^{r-1}x_{j,m+r}^-\right)   \\
[\psi_{ir},x_{jm}^+]&=-\gamma^{-1/2}[a_{ij}]_i\left(\sum_{k=1}^{r-1}d^{k-1}(q_i-q_i^{-1})x_{j,m+k}^+\psi_{i,r-k}+d^{r-1}x_{j,m+r}^+\right) .
\end{align*}
\end{lem}
The anti-automorphism $\Omega$ sends $\psi_{ik}$ to $\phi_{ik}$ and since $\Omega T_i=T_i\Omega$, $\Omega$ sends $x_{i,m}^\pm$ to $x_{i,m}^\mp$ (see \cite[Section 1 and 3.11]{MR1301623} and \cite{MR94m:17016}).  Thus one gets a similar set of commutator formulas for $\phi_{ir}$, but since we don't need them explicitly we will not write them down. 

The next result follows from direct calculations using Lusztig's elements. 
\begin{prop} \label{Arelns}
The following commutation relations hold between the generators
of $U_{\mathbb A}$.  For $k\in\mathbb Z$, $l\in\mathbb Z_{> 0}$, $1\leq i,j\leq N$,
\begin{align*}
x_{ik}^+ \Lnum{K_j}{s}{n} & = \Lnum{K_j}{s-a_{ji}}{n}x_{ik}^+, \\
x_{ik}^+ \Lnum{D}{s}{n} &= \Lnum{D}{s-k}{n}x_{ik}^+, \\
\Lnum{K_j}{s}{n}x_{ik}^- &= x_{ik}^- \Lnum{K_j}{s-a_{ji}}{n}, \\
\Lnum{D}{s}{n}x_{ik}^- &= x_{ik}^- \Lnum{D}{s-k}{n},\\
\Lnum{D}{s}{n}h_{ik} &= h_{ik} \Lnum{D}{s+k}{n},\\
[\gamma^{\pm 1/2},U_{\mathbb A} ]&=[D,K_i^{\pm 1}]=[K_i,K_j]=[K_i,h_{jk}]=0, \\
Dh_{ir} &=q^rh_{ir}D,\quad Dx_{ir}^{\pm} =q^rx_{ir}^{\pm}D,\\
K_ix_{jr}^{\pm} &= q_i^{\pm  (\alpha_i|\alpha_j)}x_{jr}^{\pm}K_i,   \\  
[h_{ik},h_{jl}]&=\delta_{k,-l} \frac{1}{k}[ka_{ij}]_i\Lnum{\gamma}{k}{1}_i \\
[h_{ik},x^{\pm}_{jl}]&= \pm \frac{1}{k}[ka_{ij}]_i\gamma^{\mp |k|/2}x^{\pm}_{j,k+l},  \\
[x^+_{ik},x^-_{jl}]&=\delta_{ij}
    \Lnum{\gamma\psi_i}{k,l}{1}.
\end{align*}

\end{prop}
%	\begin{proof}
%	These relations all follow from the defining relations of $U_q(\widehat{\widehat{\mathfrak g}})$, and the definition of the Lusztig elements including $ \Lnum{\gamma}{s}{n}_i$ and $\Lnum{\gamma\psi_i}{k,l}{1}$.   Note that by \eqnref{phidef}, the elements $\Lnum{\gamma\psi_i}{k,l}{1}$ are in $U_q(\widehat{\mathfrak g})$.    For example the first identity above gives us
%	\begin{align*}
%	x_{ik}^+\prod_{r=1}^n
%	\frac{K_jq_j^{s-r+1} - K_j^{-1}q_j^{-(s-r+1)}}{q_j^r-q_j^{-r}}
%	&=\prod_{r=1}^n\frac{K_jq_j^{s-r+1}q^{-(\alpha_i|\alpha_j)} - K_j^{-1}q_j^{-(s-r+1)}q^{(\alpha_i|\alpha_j)}}{q_j^r-q_j^{-r}}x_{ik}^+ \\
%	&=\prod_{r=1}^n\frac{K_jq_j^{s-r+1-a_{ji}} - K_j^{-1}q_j^{-(s-r+1-a_{ji})} }{q_j^r-q_j^{-r}}x_{ik}^+
%	\end{align*}
%	due to the fact that $(\alpha_i|\alpha_j)=d_ia_{i,j}=d_ja_{j,i}$.

%	 \end{proof}
 
\begin{cor} \label{Arelnscor}The algebra $U_{\mathbb A}$
inherits the standard triangular decomposition of $U_q(\hat{\mathfrak g})$.  
In particular,
any element $u $ of $U_{\mathbb A}$ can be written as an $\mathbb A$-linear combination of
monomials of the form $u^-u^0u^+$ where $u^{\pm} \in U_{\mathbb A}^{\pm}$
and $u^0 \in U_{\mathbb A}^0$.
\end{cor}

\begin{proof} We first consider an element of the form $wx_{jl}^-$ where $w$ is a monomial in $U^0_{\mathbb A}$ or $U^+_{\mathbb A}$. We need to move the $w$ past $x^-_{jl}$.  
If $w=x_{il}^+$, then
\begin{align*}
x_{ik}^+x_{jl}^-=x_{jl}^-x_{ik}^++\delta_{ij}
    \Lnum{\gamma\psi_i}{k,l}{1},
\end{align*}
which now has summands that are in $U_{\mathbb A}^-U_{\mathbb A}^0U_{\mathbb A}^+$.
Similarly if $w$ is any of the following elements
$$  
h_{is}, s>0,\enspace K_i^{\pm 1}, \enspace \gamma^{\pm 1/2},D^{\pm 1} ,  \Lnum{K_i}{s}{n},\Lnum{D}{s}{n},\Lnum{\gamma}{s}{1}_i,
$$
 then the Proposition above shows that $wx_{jl}^-\in U_{\mathbb A}^-U_{\mathbb A}^0U_{\mathbb A}^+$.  If $w= \Lnum{\gamma\psi_i}{k,l}{1}$, then \lemref{psix2} tells us (using induction on  $k$), that $wx_{jl}^-\in U_{\mathbb A}^-U_{\mathbb A}^0U_{\mathbb A}^+$.

If we consider elements of the form $wz$ where $z=h_{ik}$, $k<0$, then if $w$ is any of the elements 
$$  
x_{jl}^+,h_{js}, s>0,\enspace K_i^{\pm 1}, \enspace \gamma^{\pm 1/2},D^{\pm 1} ,  \Lnum{K_i}{s}{n},\Lnum{D}{s}{n},\Lnum{\gamma}{s}{1}_i,  \Lnum{\gamma\psi_i}{k,l}{1}
$$
then the Proposition above shows that $wh_{ik}\in U_{\mathbb A}^-U_{\mathbb A}^0U_{\mathbb A}^+$.

If we consider elements of the form $wz$ where 
$$
z= \Lnum{\gamma\psi_i}{k,l}{1}=\frac{-\gamma^{\frac{l-k}{2}}\phi_{i,k+l}}{q_i-q^{-1}_i}
$$ 
$k+l<0$, and if $w$ is any of the elements 
$$  
x_{jl}^+,h_{js}, s>0,\enspace K_i^{\pm 1}, \enspace \gamma^{\pm 1/2},D^{\pm 1} ,  \Lnum{K_i}{s}{n},\Lnum{D}{s}{n},\Lnum{\gamma}{s}{1},  \Lnum{\gamma\psi_i}{k,l}{1}
$$
then \lemref{psix2} and \propref{Arelns} above shows that $w \Lnum{\gamma\psi_i}{k,l}{1}\in U_{\mathbb A}^-U_{\mathbb A}^0U_{\mathbb A}^+$.
\end{proof}

Let $Aut(\Gamma)$ be the set of automorphisms of the afffine Dynkin diagram $\Gamma$.  Recall $I_0=\{1,...,N\}$, and let $\pi:\mathbb Z\ni r\mapsto\pi_r\in I$, $N_1,...,N_n\in\mathbb Z_{\geq 0}$, $\tau_1,...,\tau_n\in Aut(\Gamma)$ be such that: 

\begin{enumerate}[i).]
\item $N_i=\sum_{j=1}^i l(\omega_j)$ $\forall i\in I_0$ (where $\langle \omega_i,\alpha_j\rangle=\delta_{ij} $ for all $i,j\in I_0$); 

\item  $s_{\pi_1}\cdots s_{\pi_{_{N_i}}}\tau_i=\sum_{j=1}^i \omega_j$ $\forall i\in I_0$; 

\item $\pi_{r+N_n}=\tau_n(\pi_r)$ $\forall r\in\mathbb  Z$;
\end{enumerate}
\color{black}
(these conditions imply that for all $r<r^{\prime}\in\mathbb Z$ $s_{\pi_r}s_{\pi_{r+1}}\cdots s_{\pi_{r^{\prime}-1}}s_{\pi_{r^{\prime}}}$ is a reduced expression, see \cite{MR0185016} and \cite{K})

Then 
$\pi$ induces a map
$$
\mathbb Z\ni r\mapsto w_r\in W\ \ {\text{
defined\ by}}\ \ w_r=\begin{cases}s_{\pi_0}\cdots  s_{\pi_{r+1}}&{\text{if}}\ r<0,
 \\   1
&{\text{if}}\ r=0,1,\\  s_{\pi_1}\cdots  s_{\pi_{r-1}}&{\text{if}}\ r> 1,
\end{cases}
$$
\color{black}
a bijection
$$\mathbb Z\ni r\mapsto\beta_r=w_r(\alpha_{\pi_r})\in\Phi^{{\text{
re}}}_+,$$
and of course a bijection $\{\pm\}\times\mathbb Z\leftrightarrow\Phi^{{\text{
re}}}$.

For all $\alpha=\beta_r\in\Phi^{{\text{
re}}}_+$ as in \eqnref{realrootvector} the root vectors $E_{\alpha}$ can be written as: 
$$E_{\beta_r}=\begin{cases}
T_{w_r^{-1}}^{-1}(E_{\pi_r})&{\text{
if}}\ r< 0, \\
E_{\pi(0)}  & \text{if } r=0,\\
E_{\pi(1)}  & \text{if }r = 1, \\
 T_{w_r}(E_{\pi_r})&{\text{
if}}\ r> 1\\\end{cases}$$ 
and we define
$$F_{\alpha}=\Omega(E_{\alpha}).$$

For $r\in\mathbb Z$, we define
$$\beta_r^{\pm}=\begin{cases}
\pm\beta_r&{\text{
if}}\ r\leq 0\cr\mp\beta_r&{\text{
if}}\ r> 0;\end{cases}$$
then of course 
$$\{\beta_r^+|r\in\mathbb Z\}=\{m\delta+\alpha\in\Phi|m\in\mathbb Z,\alpha\in \Delta_{0,+}\},$$
$$\{\beta_r^-|r\in\mathbb Z\}=\{m\delta-\alpha\in\Phi|m\in\mathbb Z,\alpha\in \Delta_{0,+}\}.$$

The root vectors do depend on $\pi$ (for example if $a_{ij}=a_{ji}=-1$ we have $T_i(E_j)\neq T_j(E_i)$). What is independent of $\pi$ are the root vectors relative to the roots $ m\delta\pm\alpha_i$:
$$
E_{ m\delta+\alpha_i}=T_{\omega_i}^{-m}(E_i),\ \ E_{ m\delta-\alpha_i}=T_{\omega_i}^{m}T_i^{-1}(E_i).$$

\begin{lem}\label{trealrootvector}
Let $\omega=\omega_1+\cdots+\omega_n$ and let
$m\in\mathbb Z_{\geq 0},\alpha\in \Delta_{0,+}, h\in\mathbb Z$. Then
$m\delta+\alpha\in\Delta_+$ implies 
$$
T_{\omega}^h(E_{m\delta+\alpha})=\begin{cases}
E_{(m-h\langle \omega,\alpha\rangle)\delta+\alpha}&{\text{if}}\ m-h\langle \omega,\alpha\rangle\geq 0\cr
-F_{(h\langle \omega,\alpha\rangle-m)\delta-\alpha}K_{(h\langle \omega,\alpha\rangle-m)\delta-\alpha}&{\text{if}}\ m-h\langle \omega,\alpha\rangle< 0;\end{cases}$$
and
$m\delta-\alpha\in\Delta_+$, implies
$$
T_{\omega}^h(E_{m\delta-\alpha})=\begin{cases}
E_{(m+h\langle \omega,\alpha\rangle)\delta-\alpha}&{\text{if}}\ m+h\langle \omega,\alpha\rangle>0\cr
-K_{-(m+h\langle \omega,\alpha\rangle)\delta+\alpha}^{-1}F_{-(m+h\langle \omega,\alpha\rangle)\delta+\alpha}&{\text{if}}\ m+h\langle \omega,\alpha\rangle\leq 0.\end{cases}$$
\end{lem}

\begin{proof}
The function $\pi$ is defined to have the following property:\
\begin{itemize}
\item If 
$m\delta+\alpha\in \Delta_+$, then there exists $r\leq 0$ such that $m\delta+\alpha=\beta_r=s_{\pi_0}s_{\pi_{-1}}\cdots s_{\pi_{r+1}}(\alpha_{\pi_r})$ (if $r= 0$, then $\beta_0=\alpha_{\pi_0}$);  
\item If 
$m\delta-\alpha\in \Delta_+$, then there exists $r\geq 1$ such that $m\delta-\alpha=\beta_r=s_{\pi_{1}}s_{\pi_{2}}\cdots s_{\pi_{r-1}}(\alpha_{\pi_r})$ (if $r= 1$, then $\beta_1=\alpha_{\pi_1}$) .
\end{itemize}
Recall $\omega(m\delta\pm \alpha)=m\delta\pm (\alpha-\langle \omega,\alpha\rangle\delta)$.  
By the definition of $\pi$ we have $\omega=s_{\pi_1}\cdots s_{\pi_{N_n}}\tau_n=\tau_ns_{\tau^{-1}_n(\pi_1)}\cdots s_{\tau^{-1}_n(\pi_{N_n})}=\tau_ns_{\pi_{1-N_n}}\cdots s_{\pi_0}$.  It follows that 
\begin{align*}
\omega  w_r=
\begin{cases}w_{r+N_n}\tau_n &\quad\text{if $r\geq 1$ or $r\leq -N_n$} \\
w_{r+N_n}s_{\pi_{r+N_n}}\tau_n=w_{r+N_n}\tau_n s_{\pi_r} &\quad\text{if $-N_n<r\leq 0$}, \\\end{cases}
\end{align*}
and
\begin{align*}
\omega (\beta_r)=
\begin{cases}\beta_{r+N_n} &\quad\text{if $r\geq 1$ or $r\leq -N_n$},\\
-\beta_{r+N_n} &\quad\text{if $-N_n<r\leq 0$}, \\\end{cases}
\end{align*}
This means $\omega(\beta_r^\pm)=\beta_{r+N_n}^\pm$.  Moreover
\begin{itemize}
\item If $r\geq 1$, then $T_\omega T_{w_r}=T_{w_r+N_n}\tau_n$.
\item If $r\leq 0$, then $$T_\omega T_{w_r^{-1}}^{-1}=\begin{cases}  T_{w^{-1}_{r+N_n}}^{-1}\tau_n &\quad \text{ if }r\leq-N_n, \\
T_{w_r+N_n}T_{\pi_{r+N_n}}\tau_n &\quad \text{ if } -N_n<r\leq 0.
 \end{cases}$$
\end{itemize}
Thus
\begin{align*}
&T_\omega(E_{\beta_r}) \\
&=\begin{cases}
E_{\beta_{r+N_n}} &\quad \text{if $r\geq 1$ or $r\leq -N_n$}, \\
T_{w_{r+N_n}}T_{\pi_{r+N_n}}(E_{\pi_{r+N_n}}) 
=-F_{\beta_{r+N_n}}K_{\beta_{r+N_n}}&\quad \text{if $ -N_n<r\leq 0$}.
\end{cases}
\end{align*}
Now assuming the result above for $h> 0$,
\begin{align*}
T_{\omega}^{h+1}(E_{m\delta+\alpha})&=\begin{cases}
 E_{(m-(h+1)\langle \omega,\alpha\rangle)\delta+\alpha}&{\text{if}}\ m-h\langle \omega,\alpha\rangle\geq 0 , \\ 
-T_{\omega}(F_{(h\langle \omega,\alpha\rangle-m)\delta-\alpha})T_{\omega}(K_{(h\langle \omega,\alpha\rangle-m)\delta-\alpha})&{\text{if}}\ m-h\langle \omega,\alpha\rangle< 0;\end{cases}  \\
&=\begin{cases}
 E_{(m-(h+1)\langle \omega,\alpha\rangle)\delta+\alpha}&{\text{if}}\ m-h\langle \omega,\alpha\rangle\geq 0 , \\ 
-\Omega T_{\omega}(E_{(h\langle \omega,\alpha\rangle-m)\delta-\alpha}) K_{((h+1)\langle \omega,\alpha\rangle-m)\delta-\alpha}&{\text{if}}\ m-h\langle \omega,\alpha\rangle< 0;\end{cases}  \\ 
%%%%%%%%%%%%%%%%%%%%%%%%%%%%%%%%
%&=\begin{cases}
% E_{(m-(h+1)\langle \omega,\alpha\rangle)\delta+\alpha}&{\text{if}}\ m-h\langle \omega,\alpha\rangle\geq 0 , \\ 
%-\Omega (E_{((h+1)\langle \omega,\alpha\rangle-m)\delta-\alpha}) K_{((h+1)\langle \omega,\alpha\rangle-m)\delta-\alpha}&{\text{if}}\ (h+1)\langle \omega,\alpha\rangle-m >  0, \\
%\Omega  (K^{-1}_{-((h+1)-\langle \omega,\alpha\rangle)\delta+\alpha}F_{-((h+1)\langle \omega,\alpha\rangle-m)\delta+\alpha})) K_{((h+1)\langle \omega,\alpha\rangle-m)\delta-\alpha}&{\text{if}}\ ( h+1)\langle \omega,\alpha\rangle-m \leq   0;  \\
%\end{cases}  \\
&=\begin{cases}
 E_{(m-(h+1)\langle \omega,\alpha\rangle)\delta+\alpha}&{\text{if}}\ m-h\langle \omega,\alpha\rangle\geq 0 , \\ 
-F_{((h+1)\langle \omega,\alpha\rangle-m)\delta-\alpha} K_{((h+1)\langle \omega,\alpha\rangle-m)\delta-\alpha}&{\text{if}}\ (h+1)\langle \omega,\alpha\rangle-m >  0,  \\
 E_{-((h+1)\langle \omega,\alpha\rangle-m)\delta+\alpha} &{\text{if}}\ ( h+1)\langle \omega,\alpha\rangle-m \leq   0,  \\
\end{cases}  \\
&=\begin{cases}
 E_{(m-(h+1)\langle \omega,\alpha\rangle)\delta+\alpha}&{\text{if}}\ m-(h+1)\langle \omega,\alpha\rangle\geq 0 , \\ 
-F_{((h+1)\langle \omega,\alpha\rangle-m)\delta-\alpha} K_{((h+1)\langle \omega,\alpha\rangle-m)\delta-\alpha}&{\text{if}}\  m- (h+1)\langle \omega,\alpha\rangle<0.
\end{cases}  \\
\end{align*}

\end{proof}

Let $m\in\mathbb Z$, $\alpha\in \Delta_{0,+}$ be such that $m\delta\pm\alpha\in\Delta$; consider the following modified root vectors:

$$X_{m\delta+\alpha}=\begin{cases}
 E_{m\delta+\alpha}&{\text{
if}}\ m\geq 0,  \\
-F_{-m\delta-\alpha}K_{-m\delta-\alpha}&{\text{
if}}\ m< 0,\end{cases}
$$
$$
X_{m\delta-\alpha}=\begin{cases}
 -K_{m\delta-\alpha}^{-1}E_{m\delta-\alpha}&{\text{
if}}\ m> 0,\cr
F_{-m\delta+\alpha}&{\text{
if}}\ m\leq 0,\end{cases}$$
($\Omega(X_{m\delta\pm\alpha})=X_{-m\delta\mp\alpha}$).

Equivalently
$$X_{\beta_r^+}=\begin{cases}
 E_{\beta_r}&{\text{
if}}\ r\leq 0\cr-F_{\beta_r}K_{\beta_r}&{\text{
if}}\ r\geq 1,\end{cases}\ \ \ X_{\beta_r^-}=\begin{cases}
 F_{\beta_r}&{\text{
if}}\ r\leq 0\cr-K_{\beta_r}^{-1}E_{\beta_r}&{\text{
if}}\ r\geq 1.
\end{cases}$$

It follows from the proof of \lemref{trealrootvector} that: 
$$\forall r\in\mathbb Z\ \ 
T_{\omega}(X_{\beta_r^{\pm}})=X_{\beta_{r+N_n}^{\pm}}.
$$

\begin{cor}
If ${\mathcal X}$ is a finite subset of $\{X_{\beta}|\beta\in\Delta^{{\text{
re}}}\}$, then  there exists $h\in\mathbb Z$ such that $T_{\omega}^h({\mathcal X})\subseteq\{X_{\beta_r^{\pm}}|r\leq 0\}$.
\end{cor}
\begin{rem}
The imaginary root vectors are fixed points for the action of the $T_{\omega_i}$, and in particular of $T_{\omega}$.
\end{rem}
\begin{thm}\label{mainresult}
Given $\um:\mathbb Z\ni r\mapsto m_r\in\mathbb Z_{\geq 0}$ such that $\#\{r\in\mathbb Z|m_r\neq 0\}<\infty$ define
$$X^-(\um)=\prod_{r\in\mathbb Z}X_{\beta_r^-}^{m_r},\ \ X^+(\um)=\prod_{r\in\mathbb Z}X_{\beta_r^+}^{m_r}$$
where one chooses a fixed ordering for the products.

Given $\ul: \Delta_+(\text{\rm im})\to\mathbb Z_{\geq 0}$ such that $\#\{(r\delta,i)\in \Delta_+(\text{\rm im})|l_{(r\delta,i)}\neq 0\}<\infty$ define 
$$
E^{{\text{
im}}}(\ul)=\prod_{(r\delta,i)\in  \Delta_+(\text{\rm im})}E_{(r\delta,i)}^{l_{(r\delta,i)}},\ \ 
F^{{\text{
im}}}(\ul)=\Omega(E^{{\text{
im}}}(\ul)),
$$
where $E_{(r\delta,i)}=E^{(i)}_{r\delta}$.
Then the set 
\begin{equation}\label{imaginarypbw}
\{X^-(\um)F^{{\text{
im}}}(\ul)K_{\alpha}D^r\gamma^{s/2}E^{{\text{im}}}(\ul^{\prime})X^+(\um^{\prime})\},\quad r,s\in\mathbb Z,\quad \alpha\in Q_0
\end{equation}
is a basis of $U_q(\hat{\mathfrak g})$. 
\end{thm}
\begin{proof}
Since we are dealing with a finite set, up to applying a suitable power of $T_{\omega}$ we can suppose $m_r=m_r^{\prime}=0$ $\forall r\geq 1$ (notice that $Q\ni\gamma\mapsto\omega(\gamma)=\gamma-(\omega|\gamma)\delta\in Q$ is a bijection).

Then the elements $X^-(\um)F^{{\text{
im}}}(\ul)K_{\gamma}E^{{\text{
im}}}(\ul^{\prime})X^+(\um^{\prime})$ are elements of the ``classical'' PBW-basis, and the claim is obvious. 

%{\it As for the proof that the vectors $$\{X^-(\um)F^{{\text{
%im}}}(\ul)K_{\gamma}E^{{\text{
%im}}}(\ul^{\prime})X^+(\um^{\prime})\}$$ span the quantum algebra, as I wrote you it is a simple consequence of the Drinfeld relations. One can also prove it with an even simpler and more direct inductive argument, without using the Drinfeld realization. This would avoid proving that the vectors $\{X^{\pm}(\um)\}$ span the subalgebra generated by $\{X_{r\delta\pm\alpha_i}\}$, which becomes a consequence of your theorem (in any case both proofs are easy).}
%\vskip .3 truecm
%{\bf REMARK 8.} 
The fact that the products \eqnref{imaginarypbw} span $U_q(\hat{\mathfrak g})$ is similar to the proof of \cite[Proposition 1, Theorem 1]{MR1662112} (see \thmref{irreducibilitythm} below).
The Levendorskii-Soibelman formula \eqnref{convexity} implies that 
$$
\{X^-(\um)F^{{\text{im}}}(\ul)K_{\gamma}\},\ \  {\text{
and}}\ \ \{K_{\gamma}E^{{\text{im}}}(\ul^{\prime})X^+(\um^{\prime})\}
$$ 
span two subalgebras of $ U_q$.  Indeed consider the left set.  Then after applying the antiautomorphism $\Omega$ to the Levendorskii-Soibelman formula we get
\begin{align*} 
F_{r\delta}^{(i)}X_{\beta_s^-} -  X_{\beta_s^-}F_{r\delta}^{(i)}&=\begin{cases}F_{r\delta}^{(i)} F_{\beta_s} -   F_{\beta_s}F_{r\delta}^{(i)}&{\text{if}}\ s\leq 0,   \cr
-F_{r\delta}^{(i)}K_{\beta_s}^{-1}E_{\beta_s} +  K_{\beta_s}^{-1}E_{\beta_s}F_{r\delta}^{(i)}  &{\text{
if}}\ s\geq 1,
\end{cases} \\
%&=\begin{cases}F_{r\delta}^{(i)} F_{\beta_s} -   F_{\beta_s}F_{r\delta}^{(i)}&{\text{if}}\ s\leq 0,   \cr
%-K_{\beta_s}^{-1}F_{r\delta}^{(i)}E_{\beta_s} +  K_{\beta_s}^{-1}E_{\beta_s}F_{r\delta}^{(i)}  &{\text{
%if}}\ s\geq 1,
%\end{cases} \\
&=\begin{cases}\sum_{(r\delta,i)<\nu_1<\dots <\nu_r <\beta_s }\bar c_{\nu} F_{\nu_r}^{a_r}\cdots F_{\nu_1}^{a_1} &{\text{if}}\ s\leq 0,   \cr
K_{\beta_s}^{-1}(-F_{r\delta}^{(i)}E_{\beta_s} +  E_{\beta_s}F_{r\delta}^{(i)})  &{\text{
if}}\ s\geq 1,
\end{cases}
\end{align*}
where $\bar c_\nu\in\mathbb Q(q^{1/2})$. 

From \cite{MR1802170}, Theorem 5.3.2 (4), one has 
$$
[E_{r\delta}^{(i)},F_{k\delta-\alpha_j}]=\begin{cases}  -x_{ijr}E_{(r-k)\delta+\alpha_j}K_{\alpha_j -k\delta}  &\quad \text{ if }r\geq k ,\\ 
  x_{ijr}F_{(k-r)\delta-\alpha_j}K_{r\delta}^{-1}  &\quad \text{ if }r< k, \\ 
\end{cases}
$$
for some $x_{ijr}\in\mathbb  Q(q^{1/2})$. 
Applying $\Omega$ to the above gives us 
$$
[F_{r\delta}^{(i)},E_{k\delta-\alpha_j}]=\begin{cases}  x_{ijr}K_{-\alpha_j+k\delta}F_{(r-k)\delta+\alpha_j}  &\quad \text{ if }r\geq k ,\\ 
 - x_{ijr}K_{-r\delta}^{-1} E_{(k-r)\delta-\alpha_j} &\quad \text{ if }r< k, \\ 
\end{cases}
$$
and both of these are in the span of  $\{X^-(\um)F^{{\text{im}}}(\ul)K_{\gamma}\}$.
Hence the left hand side $F_{r\delta}^{(i)}X_{\beta_s^-} -  X_{\beta_s^-}F_{r\delta}^{(i)}$ above is in the span of  $\{X^-(\um)F^{{\text{im}}}(\ul)K_{\gamma}\}$.

Let $ U_q^{\prime}$ be the span of $$\{X^-(\um)F^{{\text{
im}}}(\ul)K_{\gamma}E^{{\text{
im}}}(\ul^{\prime})X^+(\um^{\prime})\};$$
then the previous paragraph implies that $ U_q^{\prime}$ is stable under left product by the $X^-(\um)F^{{\text{
im}}}(\ul)K_{\gamma}$'s and 
under right product by the $K_{\gamma}E^{{\text{
im}}}(\ul^{\prime})X^+(\um^{\prime})$'s.

The classical PBW-basis is a subset of 
$$\{X^-(\um)F^{{\text{
im}}}(\ul)F(\up)K_{\gamma}E(\up^{\prime})E^{{\text{
im}}}(\ul^{\prime})X^+(\um^{\prime})\}$$
where $$F(\up)=\prod_{r\in\mathbb Z_+}F_{\beta_r}^{p_r},\ \ E(\up)=\prod_{r\in\mathbb Z_+}E_{\beta_r}^{p_r}$$
with $\up:\mathbb Z_+\to\mathbb Z_{\geq 0}$, 
$\#\{r\in\mathbb Z_+|p_r\neq 0\}<\infty$.
This implies that in order to prove that $ U_q^{\prime}= U_q$ it is enough to prove that 
$F(\up)E(\up^{\prime})\in U_q^{\prime}$.
This is done by induction on the ``height'' $h$ of $F(\up)E(\up^{\prime})$: 
$$h=\sum_{r\in\mathbb Z_+}(p_r+p_r^{\prime})h(\beta_r),\ {\text{
where}}\  
h(\sum_{i\in I}m_i\alpha_i)=\sum_{i\in I}m_i,$$ the cases $\up\equiv \underline{0}$ or $\up^{\prime}\equiv \underline{0}$ being obvious. 
Recall from the definition of $U_q(\hat{\mathfrak g})$, \secref{jimbodrinfeld}, that  $F(\up)E_{\beta_k}=E_{\beta_k}F(\up)+\sum fke$ with $k$ a monomial in the $K_i$'s ($i\in I$), $f\in U_q^-$ and $e\in U_q^+$ homogeneous, and the ``height'' of $fe$ less than that of $F(\up)E_{\beta_k}$. The scholium follows from the fact that the ``height'' of $F(\up)E(\up^{\prime\prime})$ and that of 
$feE(\up^{\prime\prime})$ are strictly less than $h$, where $p_r^{\prime\prime}=p_r^{\prime}-\delta_{rk}$. Hence the theorem follows.
\end{proof}

\section{Imaginary Verma Modules}
The algebra $\g$ has a triangular decomposition
$\g = \widehat{\mathfrak g}_{-S} \oplus \widehat{\mathfrak h}\oplus \widehat{\mathfrak g}_S$, where
$\widehat{\mathfrak g}_S = \oplus_{\alpha\in S}\widehat{\mathfrak g}_{\alpha}$ and $S$ is defined in \secref{partition}.
Let $U(\widehat{\mathfrak g}_S)$ (resp. $U(\widehat{\mathfrak g}_{-S})$) denote the universal enveloping
algebra of $\widehat{\mathfrak g}_S$ (resp. $\widehat{\mathfrak g}_{-S}$).

Let $\lambda \in P$, where $P$ is the weight lattice of $\g$.
A weight (with respect to $\widehat{\mathfrak h}$) $U(\g)$-module $V$ is
called an $S$-highest weight module with highest weight $\lambda$ if there
is some nonzero vector $v\in V$ such that
\begin{enumerate}[(i).]
\item $u^+ \cdot v = 0$ for all $u^+ \in   \widehat{\mathfrak g}_{S}$;
\item  $V = U(\g)\cdot v$.
\end{enumerate}

Let $\lambda \in P$.  We make $\mathbb C$ into a 1-dimensional
$U(\widehat{\mathfrak g}_{S} \oplus\widehat{ \mathfrak h})$-module by picking a generating
vector $v$ and setting
$(x+h)\cdot v = \lambda(h)v$, for all $ x\in \widehat{\mathfrak g}_{S}, h \in \widehat{\mathfrak h}$.
The induced module
$$
M(\lambda) = U(\g) \otimes_{U(\widehat{\mathfrak g}_{S} \oplus \widehat{ \mathfrak h})}\mathbb Cv = U(\widehat{\mathfrak g}_{-S})\otimes \mathbb C v
$$
is called the {\it imaginary Verma module} with
$S$-highest weight $\lambda$.
Imaginary Verma modules are in many ways similar to ordinary
Verma modules except they contain both finite and
infinite-dimensional weight spaces.  They
were studied in \cite{MR95a:17030}, from which we
summarize.

\begin{prop}[\cite{MR95a:17030}, Proposition 1, Theorem 1]\label{irreducibilitythm}
Let $\lambda \in P$, and let $M(\lambda)$ be the imaginary Verma module
of $S$-highest weight $\lambda$.  Then $M(\lambda)$ has the following properties.
\begin{enumerate}[(i).]
\item The module $M(\lambda)$ is a free $U(\widehat{\mathfrak g}_{-S})$-module of rank 1
generated by the $S$-highest weight vector $1 \otimes 1$ of weight $\lambda$.
\item  $M(\lambda)$ has a unique maximal submodule.
\item  Let $V$ be a $U(\g)$-module generated by some
$S$-highest weight vector $v$ of weight $\lambda$.  Then there exists a
unique surjective homomorphism $\phi:M(\lambda) \mapsto V$ such
that $\phi(1\otimes  1) = v$.
\item  $\dim M(\lambda)_{\lambda} = 1$.  For any $\mu = \lambda - k\delta$, $k$
a positive integer, $0< \dim M(\lambda)_{\mu} < \infty$.  If
$\mu \neq \lambda - k \delta$ for any integer $k \ge 0$ and
$\dim M(\lambda)_{\mu} \neq 0$, then $\dim M(\lambda)_{\mu} = \infty$.
\item  Let $\lambda, \mu \in  \widehat{\mathfrak h}^*$.  Any non-zero element of
$\Hom_{U(\g)}(M(\lambda), M(\mu))$ is injective.
\item  The module $M(\lambda)$ is irreducible if and only if
$\lambda(c) \neq 0$.
\end{enumerate}
\end{prop}

\subsection{The Subalgebras $U_q(-S)$ and $U_q^-(S)$ of $U_q(\hat{\mathfrak g})$}
Let $U_q(\pm S)$ be the subalgebra of $U_q:=U_q(\g)$ generated by
$\{ X_{\beta_r^\pm}\ | r\in\mathbb Z\}\cup\{E_{\pm k\delta}^{(i)}|\,1\leq i\leq N, k>0\}$,
and let
$B_q^{r}$ denote the subalgebra of $U_q(\g)$ generated
by $U_q(S)\cup U_q^0(\g)$ (the superscript $r$ is used to remind us that it is generated in part by root vectors).
Let $U_q^\pm( S)$ be the subalgebra of $U_q(\g)$ generated by
$\{ x^\pm_{ik}\ | 1\leq i\leq N, k\in\mathbb Z\}\cup\{h_{il}|\,1\leq i\leq N, l\in\pm\mathbb Z_{\geq 0}^*\}$,
and let
$B_q^{d}$ denote the subalgebra of $U_q(\g)$ generated
by $U_q^+(S)\cup U_q^0(\g)$.  (The superscript $d$, is used to remind us that the respective subalgebras are generated in part by Drinfeld generators).

Let $\lambda \in P$.
A $U_q(\g)$ weight module $V_q^r$ is called an $S$-highest weight
module with highest weight $\lambda$ if there is a non-zero vector
$v \in V_q^r$ such that:
\begin{enumerate}[(i).]
\item $u^+\cdot v  =0$ for all $u^+ \in U_q(S) \setminus \mathbb C(q^{1/2})$;
\item $V_q^r = U_q(\g) \cdot v$.
\end{enumerate}
Note that, in the absence of a general quantum PBW theorem for
non-standard partitions, we cannot immediately claim that  an
$S$-highest weight module $V_q^r$ is generated by $U_q(-S)$.
This is in contrast to the classical case.

Let $\mathbb C(q^{1/2}) \cdot v$ be a 1-dimensional vector space over $\mathbb C(q^{1/2})$.
Let $\lambda \in P$, and set  $X_{\beta^+_r}\cdot v=0$, for all $r\in\mathbb Z$ and $E_{k\delta}^{(i)}\cdot v=0$
for $k<0$ and $1\leq i\leq N$ ,  $K_i^{\pm 1}\cdot v = q^{\pm\lambda(h_i)}v$
($i \in I$) and
$D^{\pm 1}\cdot v = q^{\pm \lambda(d)}v$.
Define $M_q^r(\lambda)=U_q(\g)\otimes_{B_q^r} \mathbb C(q^{1/2}) v$.  Then
$M_q^r(\lambda)$ is an $S$-highest weight $U_q$-module called
the {\it quantum imaginary Verma module} with highest weight $\lambda$.
If we let $L_q^r$ be the left ideal in $U_q$ generated by $X_{\beta_r^+} $ for all $r\in\mathbb Z$ and $E_{k\delta}^{(i)} $
for $k<0$ and $1\leq i\leq N$,  $K_i^{\pm 1}- q^{\pm\lambda(h_i)}$
($i \in I$) and
$D^{\pm 1}-q^{\pm \lambda(d)}$, then $U_q/L_q^r\cong M_q^r(\lambda)$ which is induced by $1\mapsto v$.

We obtain the following refinement of \cite[Theorem 3.4]{MR1662112}:
\color{black}

\begin{thm} \label{pbwthm} As a vector space, $M_q^r(\lambda)$ has a basis consisting of the ordered monomials
\begin{equation}\label{pbw}
\{X^-(\um)F^{{\text{im}}}(\ul)v\}.
\end{equation}
\end{thm}
\begin{proof}  This module is free over this set of vectors by \thmref{mainresult}.
\end{proof}
\begin{cor}  $M_q^r(\lambda)$ is free as a module over $U_q(-S)$.
\end{cor}
Recall the notation from \secref{drinfeldgen}.
Let $M^d_q(\lambda)=U_q/L_q^d$ where $L_q^d$ is the left ideal generated by the Drinfeld generators $x^+_{ik}$, $h_{il}$, $i\in  I_0$, $k\in\mathbb Z$, $l>0$, together with $K_i^{\pm 1}-q^{\pm \lambda(h_i)}$, $\gamma^{\pm 1/2}-q^{\pm \lambda(c)/2}$ and $D^{\pm 1}-q^{\pm \lambda(d)}$.    Let $B_q^d$ be the subalgebra of $U_q$ generated by $U^+_q(S)$ and $U^0_q(\hat{\mathfrak g})$ and let $\mathbb C(q^{1/2})_\lambda$ be the one dimensional $B_q^d$-module where $x^+_{ik}1=0$, $h_{il}1=0$, $K_i^{\pm 1}1=q^{\pm \lambda(h_i)}1$, $i\in  I_0$, $k\in\mathbb Z$, $l>0$,  $\gamma^{\pm 1/2}1=q^{\pm \lambda(c)/2}1$ and $D^{\pm 1}1=q^{\pm \lambda(d)}1$.  Note that $B_q^d\subseteq B_q^r$ as $E_{\alpha_i+k\delta}=o(i)^kx^+_{ik}$ for $k\geq0$,  $F_{-\alpha_i-k\delta}=-o(i)^kK_i \gamma^kx^+_{ik}$ for $k<0$, and $E_{k\delta}^{(i)}=\gamma^{-k/2}h_{ik}$ (see \eqnref{rootvector1} and  \eqnref{rootvector4}).

By universal mapping properties of quotients and the tensor products one has
$$
M_q^d(\lambda)\cong U_q\otimes _{B^d_q}\mathbb C(q^{1/2})_\lambda.
$$
Since $L^d_q\subset L^r_q$, there is a surjective $U_q$-module homomorphism $\pi: M^d_q(\lambda)\to M_q^r(\lambda)$. 
\begin{cor}  
$M^d_q(\lambda)$ is isomorphic to $M^r_q(\lambda)$ as $U_q$-modules. 
\end{cor}
\begin{proof}  Due to the previous Corollary there exists a $U_q(-S)$-module homomorphism $\psi: M_q^r(\lambda)\to M_q^d(\lambda)$ such that $\psi\circ \pi$ is the identity on $M_q^r(\lambda)$.   The image of $\psi$ is a $U_q(-S)$-submodule of $M_q^d(\lambda)$.  But $x_{ik}^-\in U_q(-S)U^0_q$, for $k\in\mathbb Z$,  and $h_{il}\in U_q(-S)U^0_q$, for $l\in -\mathbb Z_{>0}$ by \eqnref{rootvector2}, \eqnref{rootvector3} and \eqnref{imaginaryrootvectordef}.  Moreover $U_q^-(S)$ is generated by these elements and thus $M_q^d(\lambda)\supseteq U_q(-S)v=U_q(-S)U_q^0v\supseteq U_q^-(S)v=M_q^d(\lambda)$ by \corref{Arelnscor}.
\end{proof}

\def\cprime{$'$} \def\cprime{$'$}

\end{document}